\documentclass[a4paper,10pt,leqno]{amsart}
\title[Survey on $L^2$-invariants and $3$-manifolds]
{Survey on $L^2$-invariants and $3$-manifolds}
\author{L\"uck, W.}
        \address{Mathematical Institute of the Univerisity of Bonn\\
                Endenicher Allee 60\\
                53115 Bonn, Germany}
         \email{wolfgang.lueck@him.uni-bonn.de}
          \urladdr{http://www.him.uni-bonn.de/lueck}
         \date{June, 2021}

\usepackage{color}
\usepackage{pdfsync}
\usepackage{calc}
\usepackage{enumerate,amssymb}
\usepackage[arrow,curve,matrix,tips,2cell]{xy}
  \SelectTips{eu}{10} \UseTips
  \UseAllTwocells

\DeclareMathAlphabet\EuR{U}{eur}{m}{n}
\SetMathAlphabet\EuR{bold}{U}{eur}{b}{n}

\makeindex             

\theoremstyle{plain}
\newtheorem{theorem}{Theorem}[section]
\newtheorem{lemma}[theorem]{Lemma}

\newtheorem{conjecture}[theorem]{Conjecture}

\newtheorem{question}[theorem]{Question}

\theoremstyle{definition}
\newtheorem{assumption}[theorem]{Assumption}
\newtheorem{definition}[theorem]{Definition}
\newtheorem{example}[theorem]{Example}

\newtheorem{remark}[theorem]{Remark}
\newtheorem{notation}[theorem]{Notation}

{\catcode`@=11\global\let\c@equation=\c@theorem}

\newcommand{\comsquare}[8]                   
{\begin{CD}
#1 @>#2>> #3\\
@V{#4}VV @V{#5}VV\\
#6 @>#7>> #8
\end{CD}
}

\newcommand{\xycomsquare}[8]                   
{\xymatrix
{#1 \ar[r]^{#2} \ar[d]^{#4} &
#3 \ar[d]^{#5}  \\
#6\ar[r]^{#7} &
#8
}
}

\newcommand{\xycomsquareminus}[8]                      
{\xymatrix{#1 \ar[r]^-{#2} \ar[d]^-{#4} &
#3 \ar[d]^-{#5}  \\
#6\ar[r]^-{#7} &
#8
}
}

\newcommand{\calb}{\mathcal{B}}
\newcommand{\calc}{\mathcal{C}}
\newcommand{\cald}{\mathcal{D}}

\newcommand{\calf}{\mathcal{F}}
\newcommand{\calp}{\mathcal{P}}

\newcommand{\caln}{{\mathcal N}}
\newcommand{\calu}{{\mathcal U}}

\newcommand{\IC}{{\mathbb C}}

\newcommand{\IH}{{\mathbb H}}

\newcommand{\IN}{{\mathbb N}}

\newcommand{\IP}{{\mathbb P}}
\newcommand{\IQ}{{\mathbb Q}}
\newcommand{\IR}{{\mathbb R}}

\newcommand{\IZ}{{\mathbb Z}}

\newcommand{\an}{\operatorname{an}}

\newcommand{\coker}{\operatorname{coker}}
\newcommand{\colim}{\operatorname{colim}}

\newcommand{\Ext}{\operatorname{Ext}}

\newcommand{\GL}{\operatorname{GL}}

\newcommand{\id}{\operatorname{id}}

\newcommand{\im}{\operatorname{im}}

\newcommand{\pr}{\operatorname{pr}}

\newcommand{\sing}{\operatorname{sing}}
\newcommand{\SL}{\operatorname{SL}}
\newcommand{\supp}{\operatorname{supp}}

\newcommand{\tors}{\operatorname{tors}}
\newcommand{\tr}{\operatorname{tr}}

\newcommand{\vol}{\operatorname{vol}}

\newcommand{\Wh}{\operatorname{Wh}}

\newcommand{\higherlim}[3]{{\setbox1=\hbox{\rm lim}
        \setbox2=\hbox to \wd1{\leftarrowfill} \ht2=0pt \dp2=-1pt
        \mathop{\vtop{\baselineskip=5pt\box1\box2}}
        _{#1}}^{#2}#3}

\newcommand{\version}[1]                       
{\begin{center} last edited on #1\\
last compiled on \today\\
name of texfile: \jobname
\end{center}
}

\newcounter{commentcounter}

\newcommand{\squarematrix}[4]{\left( \begin{array}{cc} #1 & #2 \\ #3 &
#4
\end{array} \right)
}

\begin{document}




\typeout{---------------------------- L2approx_survey.tex ----------------------------}


\typeout{------------------------------------ Abstract
  ----------------------------------------}

\begin{abstract}
  In this paper give a survey about $L^2$-invariants focusing on $3$-manifolds.
\end{abstract}

\maketitle


\typeout{------------------------------- Section 0: Introduction
  --------------------------------}

\setcounter{section}{-1}

\section{Introduction}\label{sec:Introduction}

The theory of $L^2$-invariants was triggered by Atiyah in his paper on the $L^2$-index
theorem~\cite{Atiyah(1976)}.  He followed the general principle to consider a classical
invariant of a compact manifold and to define its analog for the universal covering taking
the action of the fundamental group into account. Its application to (co)homology, Betti
numbers and Reidemeister torsion led to the notions of $L^2$-(cohomology), $L^2$-Betti
numbers and $L^2$-torsion. Since then $L^2$-invariants were developed much further. They
already had and will continue to have 
striking, surprizing, and deep impact on questions and problems of fields, for some of
which one would not expect any relation, e.g., to algebra, differential geometry, global
analysis, group theory, topology, transformation groups, and von Neumann algebras.

The theory of $3$-manifolds has also a quite impressive history culminating in the work of
Waldhausen and in particular of Thurston and much later in the proof of the Geometrization
Conjecture by Perelman and of the Virtual Fibration Conjecture by Agol. It is amazing how
many beautiful, easy to comprehend, and deep theorems, which often represent the best
result one can hope for, have been proved for $3$-manifolds.

The motivating question of this survey paper is: What happens if these two prominent and
successful areas meet one another?  The answer will be: Something very interesting.

In Section~\ref{sec:Brief_surfey_on_3-manifolds} we give a brief overview over basics
about $3$-manifolds, which of course can be skipped by someone who has already some
background. We will explain the prime decomposition, Kneser's Conjecture, the Jaco-Shalen
splitting, Thurston's Geometrization Conjecture, and the Virtual Fibration Conjecture.
They give a deep insight into the structure of $3$-manifolds. All these conjectures have
meanwhile been proved. We also explain the Thurston norm and polytope, which
have been connected to $L^2$-invariants in the recent years.

In Section~\ref{sec:Brief_survey_on_L2-invariants} we briefly explain the definition of
$L^2$-Betti numbers and of $L^2$-torsion including the necessary input from the theory of
von Neumann algebras. For the rest of the article the concrete constructions are not
relevant and can be skipped, but one has to understand the basic properties of the
$L^2$-invariants, see
Subsection~\ref{subsec:Basic_properties_of_L2-Betti_numbers_and_L2-torsion_of_universal_coverings_of_finite_CW-complexes}.
For the large variety of applications of $L^2$-invariants, we will refine ourselves to
$3$-manifolds.  A discussion of all the other plentiful applications to various different fields goes
beyond the scope of this survey article.

In Section~\ref{sec:Some_open_conjectures_about_L2-invariants} we discuss the main open
conjectures and problems about $L^2$-invariants: the Atiyah Conjecture, the Singer
Conjecture, the Determinant Conjecture, various conjectures about approximation and
homological growth, and the relation between simplicial volume and $L^2$-invariants for
aspherical closed manifolds.  These conjectures make sense and are interesting in all
dimensions.

In general $L^2$-invariants are hard to compute. We explain in
Section~\ref{sec:The_computation_of_L2-Betti_numbers_and_L2-torsion_of_3-manifolds} that
one can compute the $L^2$-Betti numbers and the $L^2$-torsion for $3$-manifolds explicitly
exploring all the known results about $3$-manifolds mentioned above.

In Section~\ref{sec:The_status_of_the_conjectures_about_L2-invariants_for_3-manifolds} we
discuss  the status of all the conjectures mentioned in
Section~\ref{sec:Some_open_conjectures_about_L2-invariants} for $3$-manifolds.  Roughly
speaking, they are essentially all known except the conjecture about homological growth,
which is wide open also in dimension $3$.

The remaining sections are dealing with rather new developments concerning the twisting of
$L^2$-invariants with (not necessarily unitary or unimodular) finite-dimensional
representations.  The basics of this construction are presented in
Section~\ref{sec:Twisting_L2-invariants_with_finite-dimensional_representations} and
Section~\ref{sec:Twisting_L2-invariants_with_a_homomorphism_to_IR}.

All these twisted invariants make sense in all dimensions and have a great potential, but
concrete and interesting results are known so far only in dimension $3$. Again this due to
the fact that the structure of $3$-manifolds is rather special and well understood
nowadays.  This will be carried out in
Sections~\ref{sec:degree_of_the_reduced_twisted_L2-torsion_function_and_the_Thurston_norm}
and~\ref{sec:The_universal_L2-torsion_and_the_Thurston_polytope}, where the Turston
norm and polytope are linked to the  degree of the $L^2$-torsion function, universal
$L^2$-torsion, and the $L^2$-polytope.

In the final
Section~\ref{sec:Profinite_completion_of_the_fundamental_group_of_a_3-manifold} we relate
the conjecture of homological growth to the question, whether the $L^2$-torsion of an aspherical closed
$3$-manifold depends only on the profinite completion of the fundamental group.


\subsection{Acknowledgments}\label{subsec:Acknowledgements}

The paper is funded by the ERC Advanced Grant \linebreak ``KL2MG-interactions'' (no.
662400) of the second author granted by the European Research Council and by the Deutsche
Forschungsgemeinschaft (DFG, German Research Foundation) under Germany's Excellence
Strategy \--- GZ 2047/1, Projekt-ID 390685813.

The author wants to thank heartily Clara L\"oh, Stefan Friedl, and the referee for their useful comments.

The paper is organized as follows:

\tableofcontents


\typeout{------------------------------- Section 1: Brief Survey on $3$-manifolds
  -------------------}

\section{Brief survey on $3$-manifolds}\label{sec:Brief_surfey_on_3-manifolds}

In this section we give a brief survey about $3$-manifolds.

For the remainder of this paper $3$-manifold is to be understood to be connected compact
and orientable, and we allow a non-empty boundary.  The assumption orientable will make
the formulation of some results easier and is no real constraint for $L^2$-invariants,
since these are multiplicative under finite coverings and therefore one may pass to the
orientation covering in the non-orientable case.


\subsection{The prime decomposition and Kneser's Conjecture}%
\label{subsec:The_prime_decomposition_and_Knesers_Conjecture}

A $3$-manifold $M$ is \emph{prime} if for any decomposition of $M$ as a connected sum
$M_1 \# M_2$, $M_1$ or $M_2$ is homeomorphic to $S^3$. It is \emph{irreducible} if every
embedded $2$-sphere bounds an embedded $3$-disk. Any prime $3$-manifold is either
irreducible or is homeomorphic to $S^1 \times S^2$~\cite[Lemma~3.13]{Hempel(1976)}.

Every $3$-manifold $M$ has a prime decomposition, i.e., one can write $M$ as a connected
sum
\begin{eqnarray*}
  M  & = & M_1 \# M_2 \# \ldots \# M_r,
\end{eqnarray*}
where each $M_j$ is prime, and this prime decomposition is unique up to renumbering and
orientation preserving homeomorphism~\cite[Theorems~3.15 and~3.21]{Hempel(1976)}.

Let $M$ be a $3$-manifold with incompressible boundary whose fundamental group admits a
splitting $\alpha\colon \pi_1(M)\to \Gamma_1 \ast \Gamma_2$.  Kneser's Conjecture, whose
proof can be found in~\cite[Chapter~7]{Hempel(1976)}, says that there are manifolds $M_1$
and $M_2$ with $\Gamma_1$ and $\Gamma_2$ as fundamental groups and a homeomorphism
$M \to M_1 \# M_2$ inducing $\alpha$ on the fundamental groups. Here \emph{incompressible
  boundary} means that no boundary component is $S^2$ and the inclusion of each boundary
component into $M$ induces an injection on the fundamental groups.

\subsection{The Jaco-Shalen-Johannson splitting}%
\label{subsec:The_Jaco-Shalen-Johannson_splitting}

We use the definition of \emph{Seifert manifold} given in~\cite{Scott(1983)}, which we
recommend as a reference on Seifert manifolds.
The work of Casson and Gabai shows that an irreducible $3$-manifold with infinite fundamental group $\pi$ is
Seifert if and only if $\pi$ contains a normal infinite cyclic subgroup, 
see~\cite[Corollary~2~on~page~395]{Gabai(1991)}. This together with the argument appearing
in~\cite[page~436]{Scott(1983)} implies the following statement: If a $3$-manifold $M$ has infinite
fundamental group and empty or incompressible boundary, then it is Seifert if and only if
it admits a finite covering $\overline{M}$, which is the total space of a $S^1$-principal
bundle over a compact orientable surface.

Johannson~\cite{Johannson(1979)} and Jaco and Shalen~\cite{Jaco-Shalen(1979)} have shown
for an irreducible $3$-manifold $M$ with incompressible boundary the following result. There is a finite
family of disjoint, pairwise-nonisotopic incompressible tori in $M$, which are not isotopic
to boundary components and which split $M$ into pieces that are Seifert manifolds or are
\emph{geometrically atoroidal}%
\index{atoroidal!geometrically}, meaning that they admit no embedded incompressible torus
(except possibly parallel to the boundary).  A minimal family of such tori is unique up to
isotopy, and we will say that it gives a \emph{toral splitting} of $M$.

A \emph{graph manifold} is an irreducible $3$-manifold for which all its pieces in the
Jaco-Shalen-Johannson splitting are Seifert fibered spaces.


\subsection{Thurston's Geometrization Conjecture}%
\label{subsec:Thurstons_Geometrization_Conjecture}

Recall that a manifold (possible with boundary) is called \emph{hyperbolic}%
\index{manifold!hyperbolic} if its interior admits a complete Riemannian metric whose
sectional curvature is constant $-1$.

\emph{Thurston's Geometrization Conjecture} for irreducible $3$-manifolds with infinite
fundamental groups states that the geometrically atoroidal pieces in the
Jaco-Shalen-Johannson splitting carry a hyperbolic structure.

Roughly speaking, a \emph{geometry}%
\index{geometry on a $3$-manifold} on a $3$-manifold $M$ is a complete locally homogeneous
Riemannian metric on its interior. The precise definition is given for instance
in~\cite[page~17]{Aschenbrenner-Friedl-Wilton(2015)}.
The universal cover of the interior has a complete
homogeneous Riemannian metric, meaning that the isometry group acts
transitively~\cite{Singer(1960)}.  Thurston has shown that there are precisely eight
simply connected $3$-dimensional geometries having compact quotients, namely $S^3$,
$\IR^3$, $S^2 \times \IR$, $\IH^2 \times \IR$, $\operatorname{Nil}$,
$\widetilde{\operatorname{SL}_2(\IR)}$, $\operatorname{Sol}$ and $\IH^3$. If a closed
$3$-manifold admits a geometric structure modelled on one of these eight geometries, then
the geometry involved is unique.

Let $M$ be a closed Seifert manifold. Then it has a geometry.  In terms of the Euler class
$e$ of the Seifert bundle and the Euler characteristic $\chi$ of the base orbifold, the
geometric structure is determined as follows~\cite[Theorem~5.3]{Scott(1983)}
\[
  \begin{array}{c|ccc} &\chi > 0 & \chi =0 & \chi < 0 \\ \hline
    e =0  & S^2 \times \IR & \IR^3 & \IH^2 \times \IR \\
    e \ne 0 & S^3 & \operatorname{Nil} & \widetilde{\operatorname{SL}_2(\IR)}
  \end{array}
\]
The geometry is $S^3$ if and only if $\pi_1(M)$ is finite. Moreover, $M$ is finitely
covered by the total space $\overline{M}$ of an $S^1$-principal bundle
$p \colon \overline{M} \to F$ over an orientable closed surface $F$. We have $e = 0$ if
and only if $e(p)=0$, and the Euler characteristic $\chi$ of the base orbifold of $M$ is
negative, zero or positive if and only if $\chi(\overline{M}/S^1)$ has the same property,
see~\cite[page~426,~427 and~436]{Scott(1983)}.

For completeness we mention that Thurston's Geometrization Conjecture implies for a closed
$3$-manifold with finite fundamental group that its universal covering is homeomorphic to
$S^3$, the fundamental group of $M$ is a subgroup of $SO(4)$ and the action of it on the
universal covering is conjugated by a homeomorphism to the restriction of the obvious
$SO(4)$-action on $S^3$.  This implies, in particular, the \emph{Poincar\'e Conjecture}
that any homotopy $3$-sphere is homeomorphic to $S^3$.

Many results about $3$-manifolds have as hypothesis that Thurston's Geometrization
Conjecture holds. Meanwhile Thurston's Geometrization Conjecture is known to be true, a
proof is given in~\cite{Kleiner-Lott(2008),Morgan-Tian(2014)} following the spectacular
ideas of Perelman.


\subsection{The Virtual Fibration Conjecture}%
\label{subsec:The_Virtual_Fibration_Conjecture}

Given a $3$-manifold and a non-trivial element
$\phi\in H^1(N;\mathbb{Q}) = \mbox{Hom}(\pi_1(N),\mathbb{Q})$, we say that \emph{$\phi$ is
  fibered} if there exists a locally trivial fiber bundle $p\colon N\to S^1$ with a
compact surface as fiber and an element $r\in \IQ$ such that the induced map
$p_*\colon \pi_1(N)\to \pi_1(S^1)=\mathbb{Z}$ coincides with $r\cdot \phi$.  We say
\emph{$\phi \in H^1(N;\mathbb{R})$ is quasi-fibered} if $\phi$ is the limit in
$H^1(N;\mathbb{R})$ of fibered classes in $H^1(N;\IQ) \subseteq H^1(N;\mathbb{R})$.  A
group is \emph{residually finite rationally solvable} ($\operatorname{RFRS}$ for short),
if there is a filtration of $\pi$ by subgroups
$\pi=\pi_0\supseteq \pi_1 \supseteq \pi_2\supseteq \cdots $ such that
\begin{enumerate}
\item $\bigcap_i \pi_i=\{1\}$;
\item for any $i$ the group $\pi_i$ is a normal, finite-index subgroup of $\pi$;
\item for any $i$ the map $\pi_i\to \pi_i/\pi_{i+1}$ factors through
  $\pi_i\to H_1(\pi_i;\IZ)/\mbox{torsion}$.
\end{enumerate}

If (P) is a property of groups, for instance being ($\operatorname{RFRS})$,
then  a group  is called virtually (P) if it contains a  subgroup of finite index which has property (P).
The following is a straightforward consequence of the Virtual Fibering Theorem of
Agol~\cite[Theorem~5.1]{Agol(2008)}, see also~\cite[Corollary~5.2]{Friedl-Vidussi(2015)}
and~\cite{Kielak(2020fibring)}.

\begin{theorem}\label{thm:quasifib}
  Let $N$ be a prime 3-manifold.  Suppose that $\pi_1(N)$ is virtually
  $\operatorname{RFRS}$.  Then there exists a finite regular cover
  $p\colon \widehat{N}\to N$ such that for every class $\phi\in H^1(N;\IR)$ the class
  $p^*\phi\in H^1(\widehat{N};\IR)$ is quasi-fibered.
\end{theorem}

The following theorem was proved by Agol~\cite{Agol(2013)} and
Wise~\cite{Wise(2012raggs),Wise(2012hierachy)} in the hyperbolic case.  It was proved by
Liu~\cite{Liu(2013)} and Przytycki-Wise~\cite{Przytycki-Wise(2014)} for graph manifolds
with boundary and it was proved by Przytycki-Wise~\cite{Przytycki-Wise(2012)} for
manifolds with a non-trivial Jaco-Shalen-Johannson decomposition and at least one
hyperbolic piece in the JSJ decomposition.

\begin{theorem}\label{thm:virtrfrs}
  If $N$ is a prime $3$-manifold that is not a closed graph manifold, then $\pi_1(N)$ is
  virtually $\operatorname{RFRS}$.
\end{theorem}

This implies that any hyperbolic $3$-manifold $M$ has a finite covering
$p \colon \overline{M} \to M$ such that $\overline{M}$ fibers over $S^1$ in the sense that
there exists a locally trivial fiber bundle $\overline{M} \to S^1$ with a compact
$2$-manifold as fiber.


\subsection{Topological rigidity}%
\label{subsec:Topological_rigidity}

The fundamental group plays a dominant role in the theory of $3$-manifolds. Besides
Kneser's Conjecture this is illustrated by the following discussion of topological
rigidity.

By the Sphere Theorem~\cite[Theorem~4.3]{Hempel(1976)}, an irreducible 3-manifold is
\emph{aspherical}, i.e., all its higher homotopy groups vanish, if and only if it is a
3-disk or has infinite fundamental group. If $M$ and $N$ are two aspherical closed
$3$-manifolds, then they are homeomorphic if and only if their fundamental groups are
isomorphic.  Actually, every isomorphism between their fundamental groups is induced by a
homeomorphism.  More generally, every $3$-manifold $N$ with torsionfree fundamental group
group is topologically rigid in the sense that any homotopy equivalence of closed
$3$-manifolds with $N$ as target is homotopic to a homeomorphism.  This follows from
results of Waldhausen, see Hempel~\cite[Lemma~10.1 and Corollary~13.7]{Hempel(1976)} and
Turaev~\cite{Turaev(1988)}, as explained for instance~\cite[Section~5]{Kreck-Lueck(2009nonasph)}.


\subsection{On the fundamental groups of $3$-manifolds}%
\label{subsec:On_the_fundamental_groups_of_3-manifolds}

The fundamental group of a closed manifold is finitely presented.  Fix a natural number
$d \ge 4$. Then a group $G$ is finitely presented if and only if it occurs as fundamental
group of a closed orientable $d$-dimensional manifold. This is not true in dimension $3$.
A detailed exposition about the problem, which finitely presented groups occur as
fundamental groups of closed $3$-manifolds, can be found
in~\cite{Aschenbrenner-Friedl-Wilton(2015)}. For us it will be important that the
fundamental group of any $3$-manifold is residually finite, This follows from~\cite{Hempel(1987)}
and the proof of the Geometrization Conjecture. More information about fundamental groups of
$3$-manifolds can be found for instance in~\cite{Aschenbrenner-Friedl-Wilton(2015)}.


\subsection{The Thurston norm and the dual Thurston polytope}%
\label{subsec:The_Thurston_norm_and_the_dual_Thurston_polytope}

Let $M$ be a compact oriented $3$-manifold.  Recall the definition
in~\cite{Thurston(1986norm)} of the \emph{Thurston norm} $x_M(\phi)$ of a $3$-manifold $M$
and an element $\phi\in H^1(M;\IZ)=\operatorname{Hom}(\pi_1(M),\IZ)$:
\[
  x(\phi)–:=\min \{ \chi_-(F)\, | \, F \subset M \textup{ properly embedded surface dual
    to }\phi\},
\]
where, given a surface $F$ with connected components $F_1, F_2, \ldots , F_k$, we define
\[\chi_-(F)~=~\sum_{i=1}^k \max\{-\chi(F_i),0\}.\]

Thurston~\cite{Thurston(1986norm)} showed that this defines a seminorm on
$H^1(M;\mathbb{Z})$ which can be extended to a seminorm on $H^1(M;\mathbb{R} )$ which we
also denote by $x_M$.  In particular we get for $r \in \IR$ and $\phi \in H^1(M;\IR)$
\begin{eqnarray}
  x_M(r \cdot \phi) 
  & = & 
        |r| \cdot x_M(\phi).
        \label{scaling_Thurston_norm}
\end{eqnarray}
If $p \colon \widetilde{M} \to M$ is a finite covering with $n$ sheets, then
Gabai~\cite[Corollary~6.13]{Gabai(1983)} showed that
\begin{eqnarray}
  x_{\widetilde{M}}(p^*\phi) 
  & = & 
        n \cdot x_M(\phi).
        \label{finite_coverings_Thurston_norm}
\end{eqnarray}
If $F \to M \xrightarrow{p} S^1$ is a fiber bundle for a $3$-manifold $M$ and compact
surface $F$, and $\phi \in H^1(M;\IZ)$ is given by
$H_1(p) \colon H_1(M) \to H_1(S^1)=\IZ$, then by~\cite[Section~3]{Thurston(1986norm)} we
have
\begin{eqnarray}
  x_M(\phi) & = & 
                  \begin{cases}
                    - \chi(F), & \text{if} \;\chi(F) \le 0;
                    \\
                    0, & \text{if} \;\chi(F) \ge 0.
                  \end{cases}
                         \label{x_for_fiber_bundles}
\end{eqnarray}

We refer to
\begin{equation}
  B_{x_M}\;  := \; \{\phi\in H^1(M;\IR)\,|\, x_M(\phi)\leq 1\}
  \label{Thurston_ball}
\end{equation}
as the \emph{Thurston norm ball}.  In the sequel we will identify
$H^1(M;\IR) = H_1(M;\IR)^*$ and $V = V^{**}$ for a finite-dimensional real representation $V$
by the obvious isomorphisms. Then there is a
notion of a polytope dual to $B_{x_M}^*$,
see~\cite[Subsection~3.5]{Friedl-Lueck(2017universal)} and we define the \emph{dual
  Thurston polytope}
\begin{equation}
  T(M)^*:=B_{x_M}^*\subset (H^1(M;\IR))^*= H_1(M;\IR).
  \label{dual_Thurston_polytope} 
\end{equation}
Explicitly it is given by
\[
  T(M)^* = \{v \in H_1(M;\IR) \mid \phi(v) \le x_M(\phi) \; \text{for all} \; \phi \in
  H^1(M;\IR)\}.
\]
Thurston~\cite[Theorem~2 on page~106 and first paragraph on page~107]{Thurston(1986norm)}
has shown that $T(M)^*$ is an integral polytope, i.e, the convex hull of finitely many
points in the integral lattice $H_1(M;\IZ)/\mbox{torsion} \subseteq H_1(M;\IR)$.

A \emph{marking} for a polytope is a (possibly empty) subset of the set of its
vertices. We conclude from Thurston~\cite[Theorem~5]{Thurston(1986norm)} that we can equip
$T(M)^*$ with a marking such that $\phi\in H^1(M;\IR)$ is fibered if and only if it pairs
maximally with a marked vertex, i.e., there exists a marked vertex $v$ of $T(M)^*$, such
that $\phi(v) >\phi(w)$ for any vertex $w\ne v$.


\typeout{------------------------ Section 2: Brief Survey on $L^2$-invariants
  ------------------------}

\section{Brief survey on $L^2$-invariants}\label{sec:Brief_survey_on_L2-invariants}


\subsection{Group von Neumann algebras}%
\label{subsec:Group_von_Neumann_algebras}

Denote by $L^2(G)$ the Hilbert space $L^2(G)$ consisting of formal sums
$\sum_{g \in G} \lambda_g \cdot g$ for complex numbers $\lambda_g$ such that
$\sum_{g \in G} |\lambda_g|^2 < \infty$.  This is the same as the Hilbert space completion
of the complex group ring $\IC G$ with respect to the pre-Hilbert space structure for
which $G$ is an orthonormal basis.  Note that left multiplication with elements in $G$
induces an isometric $G$-action on $L^2(G)$. Given a Hilbert space $H$, denote by
$\calb(H)$ the $C^*$-algebra of bounded operators from $H$ to itself, where the norm is
the operator norm and the involution is given by taking adjoints.

\begin{definition}[Group von Neumann algebra]\label{def:group_von_Neumann_algebra}
  The \emph{group von Neumann algebra} $\caln(G)$ of the group $G$ is defined as the
  algebra of $G$-equivariant bounded operators from $L^2(G)$ to $L^2(G)$
  \begin{eqnarray*}
    \caln(G)
    & := &
           \calb(L^2(G))^G.
  \end{eqnarray*}
\end{definition}

In the sequel we will view the complex group ring $\IC G$ as a subring of $\caln(G)$ by
the embedding of $\IC$-algebras $\rho_r \colon \IC G \to \caln(G)$ which sends $g \in G$
to the $G$-equivariant operator $r_{g^{-1}} \colon L^2(G) \to L^2(G)$ given by right
multiplication with $g^{-1}$.

\begin{example}[The von Neumann algebra of a finite group]\label{exa:group_von_Neumann_algebra_of_a_finite_group}
  If $G$ is finite, then nothing happens, namely $\IC G = L^2(G) = \caln(G)$.
\end{example}

\begin{example}[The von Neumann algebra of $\IZ^d$]\label{exa:group_von_neumann_algebra_of_Zd}
  In general there is no concrete model for $\caln(G)$.  However, for $G = \IZ^d$, there
  is the following illuminating model for the group von Neumann algebra $\caln(\IZ^d)$.
  Let $L^2(T^d)$ be the Hilbert space of equivalence classes of $L^2$-integrable
  complex-valued functions on the $d$-dimensional torus $T^d$, where two such functions
  are called equivalent if they differ only on a subset of measure zero.  Define the ring
  $L^{\infty}(T^d)$ by equivalence classes of essentially bounded measurable functions
  $f\colon T^d \to \IC$, where essentially bounded means that there is a constant $C > 0$
  such that the set $\{x \in T^d\mid |f(x)| \ge C\}$ has measure zero. An element
  $(k_1, \ldots , k_d)$ in $\IZ^d$ acts isometrically on $L^2(T^d)$ by pointwise
  multiplication with the function $T^d \to \IC$, which maps $(z_1, z_2, \ldots, z_d)$ to
  $z_1^{k_1} \cdot \cdots \cdot z_d^{k_d}$.  The Fourier transform yields an isometric
  $\IZ^d$-equivariant isomorphism $L^2(\IZ^d) \xrightarrow{\cong} L^2(T^d)$.  We conclude
  $\caln(\IZ^d) = \calb(L^2(T^d))^{\IZ^d}$.  We obtain an isomorphism of $C^*$-algebras
  \[
    L^{\infty}(T^d) \xrightarrow{\cong} \caln(\IZ^d)
  \]
  by sending $f \in L^{\infty}(T^d)$ to the $\IZ^d$-equivariant operator
  $M_f\colon L^2(T^d) \to L^2(T^d), \; g \mapsto g \cdot f,$ where $(g\cdot f)(x)$ is
  defined by $g(x)\cdot f(x)$.
\end{example}


\subsection{The von Neumann dimension}%
\label{subsec:The_von_Neumann_dimension}

An important feature of the group von Neumann algebra is its trace.

\begin{definition}[Von Neumann trace]\label{def:trace_of_the_group_von_Neumann_algebra}
  The \emph{von Neumann trace} on $\caln(G)$ is defined by
  \[
    \tr_{\caln(G)}\colon \caln(G) \to \IC, \quad f \mapsto \langle f(e),e\rangle_{L^2(G)},
  \]
  where $e \in G \subseteq L^2(G)$ is the unit element.
\end{definition}

\begin{definition}[Finitely generated Hilbert module]\label{def:finitely_generated_Hilbert_module}
  A finitely generated \emph{Hilbert $\caln(G)$-module} $V$ is a Hilbert space $V$
  together with a linear isometric $G$-action such that there exists an isometric linear
  $G$-embedding of $V$ into $L^2(G)^r$ for some natural number $r$.  A \emph{morphism of
    Hilbert $\caln(G)$-modules} $f\colon V \to W$ is a bounded $G$-equivariant operator.
\end{definition}

\begin{definition}[Von Neumann dimension]\label{def:von_Neumann_dimension}
  Let $V $ be a finitely generated Hilbert $\caln(G)$-module.  Choose a matrix
  $A = (a_{i,j}) \in M_{r,r}(\caln(G))$ with $A^2 = A$ such that the image of the
  $G$-equivariant bounded operator $r_A^{(2)} \colon L^2(G)^r \to L^2(G)^r$ given by $A$
  is isometrically $G$-isomorphic to $V$.  Define the \emph{von Neumann dimension} of $V$
  by
  \begin{eqnarray*}
    \dim_{\caln(G)}(V)
    & := & \sum_{i=1}^r \tr_{\caln(G)}(a_{i,i})
           \quad \in \IR^{\ge 0}.
  \end{eqnarray*}
\end{definition}

The von Neumann dimension $\dim_{\caln(G)}(V)$ depends only on the isomorphism class of
the Hilbert $\caln(G)$-module $V$ but not on the choice of $r$ and the matrix $A$.  The
von Neumann dimension $\dim_{\caln(G)}$ is \emph{faithful}, i.e.
$\dim_{\caln(G)}(V) = 0 \Leftrightarrow V= 0$ holds for any finitely generated Hilbert
$\caln(G)$-module $V$. It is weakly exact in the following sense, see~\cite[Theorem~1.12
on page~21]{Lueck(2002)}.

\begin{lemma}\label{lem:weak_exactness}
  Let $0 \to V_0 \xrightarrow{i} V_1 \xrightarrow{p} V_2 \to 0$ be a sequence of finitely
  generated Hilbert $\caln(G)$-modules.  Suppose that it is weakly exact, i.e., $i$ is
  injective, the closure of $i$ is the kernel of $p$ and the image of $p$ is dense.  Then
  \[
    \dim_{\caln(G)}(V_1) = \dim_{\caln(G)}(V_0) + \dim_{\caln(G)}(V_0).
  \]
\end{lemma}

\begin{example}[Von Neumann dimension for finite groups]\label{exa:Von_Neumann_dimension_for_finite_groups}
  If $G$ is finite, then $\dim_{\caln(G)}(V)$ is $\frac{1}{|G|}$-times the complex
  dimension of the underlying complex vector space $V$.
\end{example}

\begin{example}[Von Neumann dimension for $\IZ^d$]\label{exa:Von_Neumann_dimension_for_Zd}
  Let $X \subset T^d$ be any measurable set and $\chi_X \in L^{\infty}(T^d)$ be its
  characteristic function. Denote by $M_{\chi_X}\colon L^2(T^d) \rightarrow L^2(T^d)$ the
  $\IZ^d$-equivariant unitary projection given by multiplication with $\chi_X$. Its image
  $V$ is a Hilbert $\caln(\IZ^d)$-module with $\dim_{\caln(\IZ^d)}(V) = \vol(X)$.
\end{example}


\subsection{Weak isomorphisms}%
\label{subsec:Weak_isomorphisms}

A bounded $G$-equivariant operator $f \colon L^2(G)^r \to L^2(G)^s$ is called a \emph{weak
  isomorphism} if and only if it is injective and its image is dense.  If there exists a
weak isomorphism $L^2(G)^r \to L^2(G)^s$, then we must have $r = s$ by
Lemma~\ref{lem:weak_exactness}.  The following statements are equivalent for a bounded
$G$-equivariant operator $f \colon L^2(G)^r \to L^2(G)^r$, see~\cite[Lemma~1.13 on
page~23]{Lueck(2002)}:

\begin{enumerate}

\item $f$ is a weak isomorphism;

\item Its adjoint $f^*$ is a weak isomorphism;

\item $f$ is injective;

\item $f$ has dense image;

\item The von Neumann dimension of the closure of the image of $f$ is $r$.

\end{enumerate}


\subsection{The Fuglede-Kadison determinant}%
\label{subsec:The_Fuglede-Kadison_Determinant}

\begin{definition}[Spectral density function]\label{def:spectral_density_function}
  Let $f \colon V \to W$ be a morphisms of finitely generated Hilbert $\caln(G)$-modules.
  Denote by $\{E_{\lambda}^{f^*f} \mid \lambda \in \IR\}$ the (right-continuous) family of
  spectral projections of the positive operator $f^*f$.  Define the \emph{spectral density
    function of} $f$ by
  \[
    F_f \colon \IR \to \IR^{\ge 0} \quad \lambda \mapsto
    \dim_{\caln(G)}\bigl(\im(E_{\lambda^2}^{f^*f})\bigr) =  \tr_{\caln(G)}(E_{\lambda^2}^{f^*f}).
  \]
\end{definition}

The spectral density function is monotone non-decreasing and right-continuous.  We have
$F(0) = \dim_{\caln(G)}(\ker(f))$.

\begin{example}[Spectral density function for finite groups]\label{exa:spectral_density_function_for_finite_groups}
  Let $G$ be finite and $f\colon U \to V$ be a morphism of finitely generated Hilbert
  $\caln(G)$-modules, i.e., of finite-dimensional unitary $G$-representations.  Then
  $F(f)$ is the right-continuous step function whose value at $\lambda$ is the sum of the
  complex dimensions of the eigenspaces of $f^*f$ for eigenvalues $\mu \le \lambda^2$
  divided by the order of $G$, or, equivalently, the sum of the complex dimensions of the
  eigenspaces of $|f|$ for eigenvalues $\mu \le \lambda$ divided by the order of $G$.
\end{example}

\begin{example}[Spectral density function for $\IZ^d$]\label{exa:spectral_density_function_for_Zd}
  Let $G = \IZ^d$. In the sequel we use the notation and the identification
  $\caln(\IZ^d) = L^{\infty}(T^d)$ of Example~\ref{exa:group_von_neumann_algebra_of_Zd}.
  For $f \in L^{\infty}(T^d)$ the spectral density function $F(M_f)$ of
  $M_f\colon L^2(T^d) \to L^2(T^d)$ sends $\lambda$ to the volume of the set
  $\{z \in T^d \mid |f(z)| \le \lambda\}$.
\end{example}

\begin{definition}[Fuglede-Kadison determinant]\label{def:Fuglede-Kadison_determinant}
  Let $f\colon V \to W$ be a morphism of finitely generated Hilbert $\caln(G)$-modules.
  Let $F_f(\lambda)$ be the spectral density function of
  Definition~\ref{def:spectral_density_function} which is a monotone non-decreasing
  right-continuous function. Let $dF$ be the unique measure on the Borel $\sigma$-algebra
  on $\IR$ which satisfies $dF((a,b]) = F(b)-F(a)$ for $a < b$.  Define the
  \emph{Fuglede-Kadison determinant}
  \[
    {\det}_{\caln(G )}(f) \in \IR^{\ge 0}
  \]
  to be the positive real number
  \[ {\det}_{\caln(G )}(f) = \exp\left(\int_{0+}^{\infty} \ln(\lambda) \; dF\right),
  \]
  if the Lebesgue integral $\int_{0+}^{\infty} \ln(\lambda) \; dF$ converges to a real
  number, and to be  $0$ otherwise.
\end{definition}

Note that in the definition above we do not require that the source and domain of $f$
agree or that $f$ is injective or that $f$ is surjective. Our conventions imply that the
Fulgede-Kadison operator of the zero operator $0 \colon V \to W$ is $1$.

\begin{example}[Fuglede-Kadison determinant for finite groups]\label{exa:det_for_finite_groups}
  To illustrate this definition, we look at the example where $G$ is finite. We
  essentially get the classical determinant $\det_{\IC}$. Namely, let $\lambda_1$,
  $\lambda_2$, $\ldots$, $\lambda_r$ be the non-zero eigenvalues of $f^{\ast}f$ with
  multiplicity $\mu_i$. Then one obtains, if $\overline{f^{\ast}f}$ is the automorphism of
  the orthogonal complement of the kernel of $f^{\ast}f$ induced by $f^{\ast}f$,
  \begin{multline*} {\det}_{\caln(G)}(f) = \exp\left(\sum_{i = 1}^r \frac{\mu_i}{|G|}
      \cdot \ln(\sqrt{\lambda_i})\right) = \prod_{i=1}^r \lambda_i^{\frac{\mu_i}{2\cdot
        |G|}} = {\det}_{\IC}\bigl(\overline{f^{\ast}f}\big)^{\frac{1}{2\cdot |G|}},
  \end{multline*}
  where ${\det}_{\IC}\bigl(\overline{f^{\ast}f})$ is put to be $1$ of $f$ is the zero
  operator and hence $\overline{f^{\ast}f}$ is $\id \colon \{0\} \to \{0\}$.  If
  $f \colon \IC G^m \to \IC G^m$ is an automorphism, we get
  \[ {\det}_{\caln(G)}(f) = \left|{\det}_{\IC}(f)\right|^{\frac{1}{|G|}}.
  \]
\end{example}

\begin{example}[Fuglede-Kadison determinant for $\caln(\IZ^d)$]\label{exa:Fuglede-Kadison_determinant_for_G_is_Zd}
  Let $G = \IZ^d$. We use the identification $\caln(\IZ^d) = L^{\infty}(T^d)$ of
  Example~\ref{exa:group_von_neumann_algebra_of_Zd}.  For $f \in L^{\infty}(T^d)$ we
  conclude from Example~\ref{exa:spectral_density_function_for_Zd}
  \[ {\det}_{\caln(\IZ^d)}\left(M_f\colon L^2(T^d) \to L^2(T^d)\right) =
    \exp\left(\int_{T^d} \ln(|f(z)|) \cdot \chi_{\{u \in T^d\mid f(u) \not= 0\}} \;
      \operatorname{dvol}_z\right)
  \]
  using the convention $\exp(-\infty) = 0$.
\end{example}

Let $i\colon H \rightarrow G$ be an injective group homomorphism.  Let $V$ be a finitely
generated Hilbert $\caln(H)$-module. There is an obvious pre-Hilbert structure on
$\IC G \otimes_{\IC H} V$ for which $G$ acts by isometries since $\IC G \otimes_{\IC H} V$
as a complex vector space can be identified with $\bigoplus_{G/H} V$. Its Hilbert space
completion is a finitely generated Hilbert $\caln (G)$-module and denoted by $i_*V$. A
morphism of finitely generated Hilbert $\caln(H)$-modules $f\colon V \to W$ induces a
morphism of finitely generated Hilbert $\caln(G)$-modules $i_*f\colon i_*V \to i_*W$.

The following theorem can be found with proof in~\cite[Theorem~3.14 on~page~128 and
Lemma~3.15~(4) on page~129]{Lueck(2002)}.

\begin{theorem}[Fuglede-Kadison determinant]\label{the:main_properties_of_det}\
  \begin{enumerate}
  \item\label{the:main_properties_of_det:composition} Let $f\colon U \to V$ and
    $g\colon V \to W$ be morphisms of finitely generated Hilbert $\caln(G)$-modules such
    that $f$ has dense image and $g$ is injective. Then
    \[ {\det}_{\caln(G)}(g \circ f) = {\det}_{\caln(G)}(f) \cdot {\det}_{\caln(G)}(g);
    \]

  \item\label{the:main_properties_of_det:additivity} Let $f_1\colon U_1 \to V_1$,
    $f_2\colon U_2 \to V_2$ and $f_3\colon U_2 \to V_1$ be morphisms of finitely generated
    Hilbert $\caln(G)$-modules such that $f_1$ has dense image and $f_2$ is
    injective. Then
    \[ {\det}_{\caln(G)}\squarematrix{f_1}{f_3}{0}{f_2} = {\det}_{\caln(G)}(f_1) \cdot
      {\det}_{\caln(G)}(f_2);
    \]

  \item\label{the:main_properties_of_det:det(f)_is_det(fast)} Let $f\colon U \to V$ be a
    morphism of finitely generated Hilbert $\caln(G)$-modules. Then
    \[ {\det}_{\caln(G)}(f) = {\det}_{\caln(G)}(f^*) = \sqrt{{\det}_{\caln(G)}(f^*f)} =
      \sqrt{{\det}_{\caln(G)}(ff^*)};
    \]

  \item\label{the:main_properties_of_det:restriction} Let $i \colon H \to G$ be the
    inclusion of a subgroup of finite index $[G:H]$. Let $i^* f\colon i^* U \to i^* V$ be
    the morphism of finitely generated Hilbert $\caln(H)$-modules obtained from $f$ by
    restriction. Then
    \begin{eqnarray*} {\det}_{\caln(H)}(i^* f) & = & {\det}_{\caln(G)}(f)^{[G:H]};
    \end{eqnarray*}

  \item\label{the:main_properties_of_det:induction} Let $i\colon H \to G$ be an injective
    group homomorphism and let $f\colon U \to V$ be a morphism of finitely generated
    Hilbert $\caln(H)$-modules.  Then
    \[ {\det}_{\caln(G)}(i_*f) = {\det}_{\caln(H)}(f).
    \]

  \end{enumerate}
\end{theorem}


\subsection{$L^2$-Betti numbers and $L^2$-torsion of finite Hilbert $\caln(G)$-chain complexes}%
\label{subsec:L2-Betti_numbers_and_L2-torsion_of_finite_caln(G)-Hilbert_chain_complexes}

Let $G$ be a group and let
\[
  \cdots 0 \to 0 \to C_n^{(2)} \xrightarrow{c_n^{(2)}} C_{n-1}^{(2)}
  \xrightarrow{c_{n-1}^{(2)}} \cdots \xrightarrow{c_2^{(2)}} C_{1}^{(2)}
  \xrightarrow{c_1^{(2)}} C_{0}^{(2)} \to 0 \to \cdots
\]
be a finite $\caln(G)$-chain complex $(C_*^{(2)},c_*^{(2)})$, i.e., each $C_p^{(2)}$ is a
finitely generated Hilbert $\caln(G)$-module, each differential $c_p^{(2)}$ is a
$G$-equivariant bounded operator and there is a natural number $n$ such that
$C_p^{(2)} = 0$ for $p < 0$ and for $p > n$.  For $p \in \IN$, we define the finitely
generated Hilbert $\caln(G)$-module
\[
  H_p^{(2)}(C_*^{(2)}) := \ker(c_p^{(2)})/\overline{\im(c_{p+1}^{(2)})}.
\]
Note that we divide out the closure of the image of $c_{p+1}^{(2)}$ to ensure that we
indeed obtain a finitely generated Hilbert $\caln(G)$-module. Denote by
\begin{eqnarray}
  b_p^{(2)}(C_*^{(2)}) & := &\dim_{\caln(G)}(H^{(2)}_p(C_*^{(2)})) \in \IR^{\ge 0}
                              \label{b_p_upper_(2)(C_ast)_Hilbert}
\end{eqnarray}
the \emph{$p$-th $L^2$-Betti number} of $C_*^{(2)}$.  We say that the complex $C_*^{(2)}$ is
\emph{$L^2$-acyclic} if all its $L^2$-Betti numbers vanish.

Define the $p$th \emph{Laplace operator}
\[
  \Delta_p^{(2)} := c_{p+1}^{(2)} \circ (c_p^{(2)})^* + (c_{p-1}^{(2)})^* \circ
  c_p^{(2)}\colon C_p^{(2)} \to C_p^{(2)}.
\]
Then $b_p^{(2)}(C_*^{(2)}) = \dim_{\caln(G)}(\ker(\Delta_p^{(2)}))$, see~\cite[Lemma~1.18
on page~24]{Lueck(2002)}.  Hence $b_p^{(2)}(C_*^{(2)})$ vanishes if and only if
$\ker(\Delta_p^{(2)})$ is trivial, or, equivalently, $\Delta_p^{(2)}$ is a weak
isomorphism.

We call $C_*^{(2)}$ of \emph{determinant class} if $\det_{\caln(G)}(c_p^{(2)}) > 0$ holds
for every $p \in \IN$.  This is equivalent to the condition that
$\det_{\caln(G)}(\Delta_p^{(2)}) > 0$ holds for every $p \in \IN$.  If $C_*$ is of
determinant class, then we define the \emph{$L^2$-torsion} of $C_*^{(2)}$ by
\begin{eqnarray}
  \rho^{(2)}(C_*^{(2)})  = \rho^{(2)}(C_*^{(2)};\caln(G))
  & := & - \sum_{p \in \IN} (-1)^p \cdot \ln({\det}_{\caln(G)}(c_p^{(2)})).
         \label{rho_upper_(2)(C_upper_(2)_ast)}
\end{eqnarray}
This turns out to be the same as putting
\begin{eqnarray}
  \rho^{(2)}(C_*^{(2)}) & = & - \frac{1}{2} \cdot \sum_{p \in \IN} (-1)^p \cdot p \cdot \ln({\det}_{\caln(G)}(\Delta_p^{(2)})).
                              \label{rho_upper_(2)(C_ast)_Laplace_Hilbert}
\end{eqnarray}


\subsection{$L^2$-Betti numbers and $L^2$-torsion of finite based free chain complexes
  over group rings}%
\label{subsec:L2-Betti_numbers_and_L2-torsion_of_chain_complexes_over_group_rings}

Let $G$ be a group and let $R$ be one of the rings $\IZ$, $\IQ$, $\IR$, or $\IC$.  Let
\[
  \cdots 0 \to 0 \to C_n \xrightarrow{c_n} C_{n-1} \xrightarrow{c_{n-1}} \cdots
  \xrightarrow{c_2} C_{1} \xrightarrow{c_1} C_{0} \to 0 \to \cdots
\]
be a finite based free $RG$-chain complex $(C_*,c_*)$, i.e., each $C_p$ is a finitely
generated free $RG$-module equipped with a $RG$-basis, each differential is a 
$RG$-homomorphism and there is a natural number $n$ such that $C_p = 0$ for $p < 0$ and
for $p > n$.  The basis induces on $C_p^{(2)} := L^2(G) \otimes_{RG} C_{p}$ the structure
of a finitely generated Hilbert $\caln(G)$-module in the obvious way.  Note for the sequel
that this structure is unchanged if we permute the basis elements or multiply one of the
basis elements with $\pm g$ for some $g \in G$.  The resulting differentials
$c_p^{(2)} = \id\otimes c_p \colon C_{i}^{(2)}\to C_{i-1}^{(2)}$ are bounded
$G$-equivariant operators.  So we get a finite Hilbert $\caln(G)$-chain complex
$C_*^{(2)}$.  For $p\in \IN$, we define the finitely generated Hilbert $\caln(G)$-module
\[
  H_p^{(2)}(C_*) := H_p^{(2)}(C_*^{(2)}).
\]
Denote by
\begin{eqnarray}
  b_p^{(2)}(C_*) = b_p^{(2)}(C_*;\caln(G)) & := &b_p^{(2)}(C_*^{(2)})  \in \IR^{\ge 0}
                                                  \label{b_p_upper_(2)(C_ast)}
\end{eqnarray}
the \emph{$p$-th $L^2$-Betti number} of $C_*$.  We say that the complex $C_*$ is
\emph{$L^2$-acyclic} if all its $L^2$-Betti numbers vanish.

We call $C_*$ of \emph{determinant class} if $C_*^{(2)} $ is of determinant class.  If
$C_*$ is of determinant class, then we define the \emph{$L^2$-torsion} of $C_*$ by
\begin{eqnarray}
  \rho^{(2)}(C_*) = \rho^{(2)}(C_*;\caln(G))  & := & \rho^{(2)}(C_*^{(2)}).
                                                     \label{rho_upper_(2)(C_ast)}
\end{eqnarray}


\subsection{$L^2$-Betti numbers and $L^2$-torsion of regular coverings of finite
  $CW$-complexes}%
\label{subsec:L2-Betti_numbers_and_L2-torsion_of_regular_coverings_of_finite_CW-complexes}

Let $G$ be a (discrete) group and $X$ be a finite $CW$-complex.  Let
$G \to \overline{X} \to X$ be a $G$-principal bundle over $X$, or, equivalently, a normal
covering with $G$ a group of deck transformations. The cellular chain complex
$C_*(\overline{X})$ of $\overline{X}$ with $\IZ$-coefficients is a finite free
$\IZ G$-chain complex. If we choose an ordering on the set of cells of $X$, an orientation
for each cell in $X$, and a lift of each cell in $X$ to cell in $\overline{X}$, we obtain
a $\IZ G$-basis for $C_*(X)$ and we can consider the finite $\caln(G)$-chain complex
$C_*^{(2)}(\overline{X})$.  One easily checks that $C_*^{(2)}(\overline{X})$ is
independent of the choices above. Hence we can define the \emph{$p$th $L^2$-Betti number}
\begin{eqnarray}
  b_p^{(2)}(\overline{X}) = b_p^{(2)}(\overline{X};\caln(G))
  & := & b_p^{(2)}(C_*^{(2)}(\overline{X})).
         \label{b_p_upper_(2)(overline(X))}                          
\end{eqnarray}
If $C_*^{(2)}(\overline{X})$ is of determinant class, we can also consider the
\emph{$L^2$-torsion}
\begin{eqnarray}
  \rho^{(2)}(\overline{X}) = \rho^{(2)}(\overline{X};\caln(G)) & := & \rho^{(2)}(C_*^{(2)}(\overline{X})).
                                                                      \label{rho_upper_(2)(overline(X))}
\end{eqnarray}

Let $X$ be a finite (not necessarily connected) $CW$-complex. Let $C$ be any of its path
components.  Let $\widetilde{C} \to C$ be the universal covering of $C$ which is a
$\pi_1(C)$-principal bundle.  So $b_p^{(2)}(\widetilde{C})$ is defined. We put
\begin{eqnarray}
  b_p^{(2)}(\widetilde{X}) & := & \sum_{C \in \pi_0(X)} b_p^{(2)}(\widetilde{C}).
                                  \label{b_p_upper_(2)(widetilde(X))}                          
\end{eqnarray}
We call $X$ \emph{of determinant class} if $C_*^{(2)}(\widetilde{C})$ is of determinant
class for every $C \in \pi_0(C)$.  In this case we put
\begin{eqnarray}
  \rho^{(2)}(\widetilde{X}) & := & \sum_{C \in \pi_0(X)} \rho^{(2)}(\widetilde{C}).
                                   \label{rho_upper_(2)(widetilde(X))}
\end{eqnarray}
We say that $\widetilde{X}$ is \emph{$L^2$-acyclic} if $b_p^{(2)}(\widetilde{X})$ vanishes
for all $p \in \IN$.  We say that $\widetilde{X}$ is \emph{$\det$-$L^2$-acyclic} if
$\widetilde{X}$ is of determinant class and $b_p^{(2)}(\widetilde{X})$ vanishes for all
$p \in \IN$.


\subsection{Basic properties of $L^2$-Betti numbers and $L^2$-torsion of universal
  coverings of finite $CW$-complexes}%
\label{subsec:Basic_properties_of_L2-Betti_numbers_and_L2-torsion_of_universal_coverings_of_finite_CW-complexes}
\ \\[1mm]
Here is a list of basic properties of $L^2$-Betti numbers and $L^2$-torsion of universal
coverings of finite $CW$-complexes.  Note that the status of the Determinant
Conjecture~\ref{con:Determinant_Conjecture} will be reviewed in
Remark~\ref{rem:status_of_Determinant_Conjecture}. It is known to be true for a very large
class of groups including sofic groups and fundamental groups of $3$-manifolds.
\begin{enumerate}

\item\label{list:main_properties_of_rho2(widetildeX):homotopy_invariance} \emph{(Simple)
    Homotopy invariance}, see~\cite[Theorem~1.35~(1) on page~37 and Theorem~3.96~(i) on
  page~163]{Lueck(2002)}.
  \\[1mm]
  Let $f\colon X \to Y$ be a homotopy equivalence of finite $CW$-complexes.
  \begin{enumerate}
  \item Then
    \[b_p^{(2)}(\widetilde{X}) = b_p^{(2)}(\widetilde{Y});
    \]
  \item Suppose that $\widetilde{X}$ or $\widetilde{Y}$ is $\det$-$L^2$-acyclic.  Then
    both $\widetilde{X}$ and $\widetilde{Y}$ are $\det$-$L^2$-acyclic;

  \item Suppose that $f$ is a simple homotopy equivalence or that $f$ is a homotopy
    equivalence and $\pi_1(X)$ satisfies the Determinant
    Conjecture~\ref{con:Determinant_Conjecture}. Assume that $\widetilde{X}$ and
    $\widetilde{Y}$ are $\det$-$L^2$-acyclic. Then
    \[
      \rho^{(2)}(\widetilde{Y}) = \rho^{(2)}(\widetilde{X});
    \]
  \end{enumerate}

\item\label{list:main_properties_of_rho2(widetildeX):Euler-Poincare_formula}
  \emph{Euler-Poincar\'e formula}, see~\cite[Theorem~1.35~(2) on page~37]{Lueck(2002)}.
  \\[1mm]
  We get for a finite $CW$-complex $X$
  \[
    \chi(X) = \sum_{p \in \IN} (-1)^p \cdot b_p^{(2)}(\widetilde{X});
  \]

\item\label{list:main_properties_of_rho2(widetildeX):sum_formula}
  \emph{Sum formula}, see~\cite[Theorem~3.96~(2)  on page~164]{Lueck(2002)}.\\[1mm]
  Consider the pushout of finite $CW$-complexes such that $j_1$ is an inclusion of
  $CW$-complexes, $j_2$ is cellular and $X$ inherits its $CW$-complex structure from
  $X_0$, $X_1$ and $X_2$
  \[
    \xymatrix{X_0 \ar[r]^-{j_1} \ar[d]_{j_2} & X_1 \ar[d]^{i_1} \\ X_2 \ar[r]_-{i_2} & X.}
  \]
  Assume that for $k=0,1,2$ the map $\pi_1(i_k,x)\colon \pi_1(X_k,x) \to \pi_1(X,j_k(x) )$
  induced by the obvious map $i_k\colon X_k \to X$ is injective for all base points $x$ in
  $X_k$.
  \begin{enumerate}
  \item If $\widetilde{X_0}$, $\widetilde{X_1}$, and $\widetilde{X_2}$ are $L^2$-acyclic,
    then $\widetilde{X}$ is $L^2$-acyclic;

  \item If $\widetilde{X_0}$, $\widetilde{X_1}$, and $\widetilde{X_2}$ are
    $\det$-$L^2$-acyclic, then $\widetilde{X}$ is $\det$-$L^2$-acyclic and we get
    \[
      \rho^{(2)}(\widetilde{X}) = \rho^{(2)}(\widetilde{X_1}) +
      \rho^{(2)}(\widetilde{X_2}) - \rho^{(2)}(\widetilde{X_0});
    \]
  \end{enumerate}

\item\label{list:main_properties_of_rho2(widetildeX):Poincare_duality}
  \emph{Poincar\'e duality}, see~\cite[Theorem~1.35~(3) on page~37 and Theorem~3.96~(3)  on page~164]{Lueck(2002)}.\\[1mm]
  Let $M$ be a closed manifold of dimension $n$
  \begin{enumerate}
  \item Then
    \[
      b_p^{(2)}(\widetilde{M}) = b_{n-p}^{(2)}(\widetilde{M});
    \]
  \item Suppose that $n$ is even and $\widetilde{M}$ is $\det$-$L^2$-acyclic. Then
    \[
      \rho^{(2)}(\widetilde{M}) = 0;
    \]
  \end{enumerate}

\item\label{list:main_properties_of_rho2(widetildeX):product_formula}
  \emph{Product formula}, see~\cite[Theorem~1.35~(4) on page~37 and Theorem~3.96~(4)  on page~164]{Lueck(2002)}.\\[1mm]
  Let $X$ and $Y$ be finite $CW$-complexes.

  \begin{enumerate}

  \item Then
    \[
      b_p^{(2)}(\widetilde{X \times Y}) = \sum_{i,j \in \IN, p = i+j}
      b_i^{(2)}(\widetilde{X}) \cdot b_j^{(2)}(\widetilde{Y});
    \]
  \item Suppose that $\widetilde{X}$ is $\det$-$L^2$-acyclic. Then
    $\widetilde{X \times Y}$ is $\det$-$L^2$-acyclic and
    \[
      \rho^{(2)}(\widetilde{X \times Y}) = \chi(Y) \cdot \rho^{(2)}(\widetilde{X});
    \]
  \end{enumerate}
\item\label{list:main_properties_of_rho2(widetildeX):multiplicativity}
  \emph{Multiplicativity}, see~\cite[Theorem~1.35~(9) on page~38 and Theorem~3.96~(3)  on page~164]{Lueck(2002)}.\\[1mm]
  Let $X \to Y$ be a finite covering of finite $CW$-complexes with $d$ sheets.

  \begin{enumerate}
  \item Then
    \[
      b_p(\widetilde{X}) = d \cdot b_p(\widetilde{Y});
    \]
  \item Then $\widetilde{X}$ is $\det$-$L^2$-acyclic if and only if $\widetilde{Y}$ is
    $\det$-$L^2$-acyclic, and in this case
    \[
      \rho^{(2)}(\widetilde{X}) = d \cdot \rho^{(2)}(\widetilde{Y});
    \]
  \end{enumerate}

\item\label{list:main_properties_of_rho2(widetildeX):sofic_det-class}
  \emph{Determinant class}.\\[1mm]
  If $\pi_1(C)$ satisfies the Determinant Conjecture~\ref{con:Determinant_Conjecture} for
  each component $C$ of the finite $CW$-complex $X$, then $\widetilde{X}$ is of
  determinant class;

\item\label{list:main_properties_of_rho2(widetildeX):det-class:b_0_upper(2)}
  \emph{$0$th $L^2$-Betti number}, see~\cite[Theorem~1.35~(8) on page~38]{Lueck(2002)}.\\[1mm]
  If $X$ is a connected finite $CW$-complex with fundamental group $\pi$, then
  \[
    b_0^{(2)}(\widetilde{X}) =
    \begin{cases}
      \frac{1}{|\pi|} & \text{if}\; \pi \; \text{is finite};
      \\
      0 & \text{otherwise;}
    \end{cases}
  \]

\item\label{list:main_properties_of_rho2(widetildeX):det-class:fibrations} \emph{Fibration
    formula}, see~\cite[Lemma~1.41 on page~45 and Corollary~3.103 on
  page~166]{Lueck(2002)}.  \smallskip
  \begin{enumerate}
  \item Let $p \colon E \to B$ a fibration such that $B$ is a connected finite
    $CW$-complex and the fiber is homotopy equivalent to a finite $CW$-complex $Z$. Suppose
    that for every $b \in B$ and $x \in F_b := p^{-1}(b)$ the inclusion $p^{-1}(b) \to E$
    induces an injection on the fundamental groups $\pi_1(F_b,x) \to \pi_1(E,x)$, and that
    $Z$ is $L^2$-acyclic.

    Then $E$ is homotopy equivalent to a finite $CW$-complex $X$ which is $L^2$-acyclic;

  \item Let $F \xrightarrow{i} E \xrightarrow{p} B$ be locally trivial fiber bundle of
    finite $CW$-complexes.  Suppose $B$ is connected, that the map
    $\pi_1(F,x) \to\pi_1(E,i(x))$ is bijective for every base point $x \in F$, and that
    $\widetilde{F}$ is $\det$-$L^2$-acyclic.
  
    Then $\widetilde{E}$ is $\det$-$L^2$-acyclic and
    \[
      \rho^{(2)}(\widetilde{E}) = \chi(B) \cdot \rho^{(2)}(\widetilde{F});
    \]
  \end{enumerate}

\item\label{list:main_properties_of_rho2(widetildeX):S_upper_1-actions}
  \emph{$S^1$-actions}, see~\cite[Theorem~1.40 on page~43 and Theorem~3.105  on page~168]{Lueck(2002)}.\\[1mm]
  Let $X$ be a connected compact $S^1$-$CW$-complex, for instance a closed smooth manifold
  with smooth $S^1$-action. Suppose that for one orbit $S^1/H$ (and hence for all orbits)
  the inclusion into $X$ induces a map on $\pi_1$ with infinite image.  (In particular the
  $S^1$-action has no fixed points.)  Then $\widetilde{X}$ is $\det$-$L^2$-acyclic and
  $\rho^{(2)}(\widetilde{M})$ vanishes;

\item\label{list:main_properties_of_rho2(widetildeX):aspherical}
  \emph{Aspherical spaces}, see~\cite[Theorem~3.111  on page~171 and Theorem~3.113  on page~172]{Lueck(2002)}.\\[1mm]
  \smallskip
  \begin{enumerate}

  \item Let $M$ be an aspherical closed smooth manifold with a smooth $S^1$-action.  Then
    the conditions appearing in
    assertion~\eqref{list:main_properties_of_rho2(widetildeX):S_upper_1-actions} are
    satisfied and hence $\widetilde{M}$ is $\det$-$L^2$-acyclic and
    $\rho^{(2)}(\widetilde{X})$ vanishes;

  \item If $X$ is an aspherical finite $CW$-complex whose fundamental group contains an
    infinite, elementary amenable, and normal subgroup, then $\widetilde{X}$ is
    $\det$-$L^2$-acyclic and $\rho^{(2)}(\widetilde{X})$ vanishes;

  \end{enumerate}

\item\label{list:main_properties_of_rho2(widetildeX):mapping_tori}
  \emph{Mapping tori}, see~\cite[Theorem~1.39 on page~42]{Lueck(2002)}.\\[1mm]
  Let $f \colon X \to X$ be a self homotopy equivalence of a finite $CW$-complex.  Denote
  by $T_f$ its mapping torus.
  
  \begin{enumerate}
  \item Then $\widetilde{T_f}$ is $L^2$-acyclic;

  \item If $\widetilde{X}$ is $\det$-$L^2$-acyclic, then $\rho^{(2)}(\widetilde{T_f})$ vanishes;
  \end{enumerate}
  
\item\label{list:main_properties_of_rho2(widetildeX):hyperbolic}
  \emph{Hyperbolic manifolds}, see~\cite{Hess-Schick(1998)},~\cite[Theorem~1.39 on page~42]{Lueck(2002)}.\\[1mm]
  Let $M$ be a hyperbolic closed manifold $M$ of dimension $n$.
  \begin{enumerate}
  \item If $n$ is odd, $\widetilde{M}$ is $\det$-$L^2$-acyclic;

  \item Suppose $n = 2m$ is even. Then $b_p^{(2)}(\widetilde{M})$ vanishes for
    $p \not= m$, and we have $(-1)^m \cdot \chi(M) = b_m^{(2)}(\widetilde{M})> 0$;

    \item For every number $m$ there exists an explicit constant $C_m> 0$ with the following
    property: If $M$ is a hyperbolic closed manifold of dimension $(2m+1)$ with volume
    $\vol(M)$, then
    \[\rho^{(2)}(\widetilde{M}) = (-1)^m \cdot C_m \cdot \vol(M).
    \]
    We have $C_1 = \frac{1}{6\pi}$. The number $\pi^m \cdot C_m$ is always rational;
  \end{enumerate}

\item\label{list:main_properties_of_rho2(widetildeX):approximation} \emph{Approximation of
    $L^2$-Betti numbers by classical Betti numbers},
  see~\cite{Lueck(1994c)},\cite[Chapter~13]{Lueck(2002)}.
  \\[1mm]
  Let $X$ be a connected finite $CW$-complex with fundamental group $G = \pi_1(X)$.
  Suppose that $G$ comes with a descending chain of subgroups
  \[
    G = G_0 \supseteq G_1 \supseteq G_2 \supseteq \cdots
  \]
  
  such that $G_i$ is normal in $G$, the index $[G:G_i]$ is finite and we have
  $\bigcap_{i \ge 0} G_i = \{1\}$.  Let $b_p(G_i\backslash \widetilde{X})$ be the $p$-th
  Betti number of the finite $CW$-complex $G_i\backslash \widetilde{X}$.

  Then $G_i\backslash \overline{X} \to X$ is a finite $[G:G_i]$-sheeted covering and we
  have
  \[
    b_p^{(2)}(\widetilde{X} ) = \lim_{i \to \infty} \frac{b_p(G_i\backslash
      \widetilde{X})}{[G:G_i]}.
  \]
\end{enumerate}

There is also  version, where the subgroups are not necessarily normal, see Farber~\cite{Farber(1998)}.

$L^2$-Betti numbers were originally defined by Atiyah~\cite{Atiyah(1976)}.
The definition of $L^2$-torsion in the analytic setting goes back to
Lott~\cite {Lott(1992a)} and Mathai~\cite{Mathai(1992)}~,
and in the topological setting to L\"uck-Rothenberg~\cite{Lueck-Rothenberg(1991)}.

For more information about $L^2$-invariants we refer for instance
to~\cite{Kammeyer(2019),Loeh(2020ergodic),Lueck(2002)}.


\typeout{---- Section 3: Some open conjectures about $L^2$-invariants -----}

\section{Some open conjectures about $L^2$-invariants}%
\label{sec:Some_open_conjectures_about_L2-invariants}

We briefly review some prominent and open conjectures about $L^2$-invariants.


\subsection{The Atiyah Conjecture}%
\label{subsec:The_Atiyah_Conjecture}

\begin{conjecture}[Atiyah Conjecture]\label{con:Atiyah_Conjecture}
  We say that a torsionfree group $G$ satisfies the \emph{Atiyah Conjecture} if for any
  matrix $A \in M_{m,n}(\IQ G)$ the von Neumann dimension $\dim_{\caln(G)}(\ker(r_A))$ of
  the kernel of the $\caln(G)$-homomorphism $r_A\colon \caln(G)^m \to \caln(G)^n$ given by
  right multiplication with $A$ is an integer.
\end{conjecture}

The Atiyah Conjecture can also be formulated for any field $F$ with
$\IQ \subseteq F \subseteq \IC$ and matrices $A \in M_{m,n}(FG)$ and for any group with a
bound on the order of its finite subgroups.  However, we only need and therefore consider
in this paper the case, where $F = \IQ$ and $G$ is torsionfree.

\begin{definition}[Admissible $3$-manifold]\label{def:admissible_3-manifold} 
  A $3$-manifold is called \emph{admissible} if it is connected, orientable, compact and
  irreducible, its boundary is empty or a disjoint union of tori, its fundamental group is
  infinite, and it is not homeomorphic to $S^1 \times D^2$.
\end{definition}

For some information about the proof and in particular of references in the literature we
refer to~\cite[Theorem~3.2]{Friedl-Lueck(2019Euler)} except for
assertion~\eqref{the:Status_of_the_Atiyah_Conjecture:locally_indicable} which is due to
Jaikin-Zapirain and Lopez-Alvarez~\cite[Proposition~6.5]{Jaikin-Zapirain+Lopez-Alvarez(2020)}.  A group is
called \emph{locally indicable} if every non-trivial finitely generated subgroup admits an epimorphism
onto $\IZ$. Examples are torsionfree one-relator groups.

\begin{theorem}[Status of the Atiyah
  Conjecture]\label{the:Status_of_the_Atiyah_Conjecture}\
  \begin{enumerate}

  \item\label{the:Status_of_the_Atiyah_Conjecture:subgroups} If the torsionfree group $G$
    satisfies the Atiyah Conjecture, see Conjecture~\ref{con:Atiyah_Conjecture},
    then also each of its subgroups satisfies the Atiyah Conjecture;

  \item\label{the:Status_of_the_Atiyah_Conjecture:Linnell} \ Let $\calc$ be the smallest
    class of groups which contains all free groups and is closed under directed unions and
    extensions with elementary amenable quotients. Suppose that $G$ is a torsionfree group
    which belongs to $\calc$.

    Then $G$ satisfies the Atiyah Conjecture;

  \item\label{the:Status_of_the_Atiyah_Conjecture:3-manifold_not_graph} Let $G$ be an
    infinite group that  is the fundamental group of an admissible $3$-manifold $M$ which
    is not a closed graph manifold.  Then $G$ is torsionfree and belongs to $\calc$. In
    particular $G$ satisfies the Atiyah Conjecture;

  \item\label{the:Status_of_the_Atiyah_Conjecture:approx} Let $\cald$ be the smallest
    class of groups such that
    \begin{itemize}

    \item The trivial group belongs to $\cald$;

    \item If $p\colon G \to A$ is an epimorphism of a torsionfree group $G$ onto an
      elementary amenable group $A$ and if $p^{-1}(B) \in \cald$ for every finite group
      $B \subset A$, then $G \in \cald$;

    \item $\cald$ is closed under taking subgroups;

    \item $\cald$ is closed under colimits and inverse limits over directed systems.

    \end{itemize}

    If the group $G$ belongs to $\cald$, then $G$ is torsionfree and the Atiyah
    Conjecture holds for $G$.

    The class $\cald$ is closed under direct sums, direct products and free
    products. Every residually torsionfree elementary amenable group belongs to $\cald$;

  \item\label{the:Status_of_the_Atiyah_Conjecture:locally_indicable} A locally indicable
    group satisfies the Atiyah Conjecture. More generally,
    if $1 \to H \to G \to Q \to 1$ is an extension of groups, $H$ satisfies the Atyiah Conjecture
   and $Q$ is locally indicable, then $G$ satisfies the Atyiah Conjecture.
  \end{enumerate}
\end{theorem}

\begin{remark}\label{rem:Atyiah_conjecture_and_closed_manifolds}
  Let $G$ be a finitely presented torsionfree group.  Then $G$ satisfies the Atiyah
  Conjecture if and only if the $p$th $L^2$-Betti number
  $b_p^{(2)}(\widetilde{M})$ is an integer for every $p \ge 0$ and every closed manifold
  $M$ with $\pi_1(M) \cong G$.
\end{remark}

\begin{remark}[Analytic version of $L^2$-Betti numbers]\label{rem:Analytic_version_of_L2-Betti_numbers}
  One can define the $L^2$-Betti number $b_p^{(2)}(\widetilde{M})$ of a closed Riemannian
  manifold $M$ by the analytic expression
  \begin{eqnarray}
    b_p^{(2)}(\widetilde{M})
    & = &
          \lim_{t \to \infty} \int_{\calf} \tr_{\IR}(e^{-t\Delta_p}(x,x)) \operatorname{dvol}.
          \label{analytic_L2-Betti_number_and_heat_kernel}
  \end{eqnarray}
  Here $e^{-t\Delta_p}(x,y)$ denotes the heat kernel for $p$-forms on the universal
  covering $\widetilde{M}$ and $\tr_{\IR}(e^{-t\Delta_p}(x,x))$ is its trace, and $\calf$
  is a \emph{fundamental domain} for the $\pi_1(M)$-action on $\widetilde{M}$, see for
  instance~\cite[Proposition~4.16 on page~63]{Atiyah(1976)}
  or~\cite[Section~1.3.2]{Lueck(2002)}.  In view of
  expression~\eqref{analytic_L2-Betti_number_and_heat_kernel} it is rather surprizing that
  this should always be an integer if the fundamental group is torsionfree.
\end{remark}

Note that any non-negative real number occurs as the von Neumann dimension of the kernel
of some morphisms of finitely generated Hilbert $\caln(G)$-modules if $G$ contains an
element of infinite order.  So it is crucial that the matrices appearing in the Atiyah
Conjecture~\ref{the:Status_of_the_Atiyah_Conjecture} live already over the group ring.

Associated to the von Neumann algebra $\caln(G)$ is the algebra of affiliated operators
$\calu(G)$ which contains $\caln(G)$.  It can be defined analytically or just as the Ore
localization of $\caln(G)$ with respect to the multiplicative subset of non-zero divisors.
Now one can consider the so called division closure $\cald(G)$ of $\IQ G$ in $\calu(G)$.

The proof of the following is based on ideas of Peter Linnell from~\cite{Linnell(1993)}
which have been explained in detail and a little bit extended in~\cite[Lemma~10.39 on
page~10.39 and Chapter~10]{Lueck(2002)} and~\cite{Reich(2006)}, see
also~\cite[Theorem~3.8]{Friedl-Lueck(2019Thurston)}.

\begin{theorem}[Main properties of $\cald(G)$]\label{the:Main_properties_of_cald(G)}
  Let $G$ be a torsionfree group.

  \begin{enumerate}

  \item\label{the:Main_properties_of_cald(G):skew_field} The group $G$ satisfies the
    Atiyah Conjecture if and only if $\cald(G)$ is a skew
    field;

  \item\label{the:Main_properties_of_cald(G):dim} Suppose that $G$ satisfies the Atiyah
    Conjecture.  Let $C_*$ be a $\IQ G$-chain complex whose
    chain-modules are finitely generated projective.  Then we get for all $n \ge 0$
    \[
      b_n^{(2)}\bigl(\caln(G) \otimes_{\IQ G} C_*\bigr) =
      \dim_{\cald(G)}\bigl(H_n(\cald(G) \otimes_{\IQ G} C_*)\bigr).
    \]
    In particular $b_n^{(2)}\bigl(\caln(G) \otimes_{\IQ G} C_*\bigr)$ is an integer.

  \end{enumerate}
\end{theorem}

Theorem~\ref{the:Main_properties_of_cald(G)} shows that the Atyiah
Conjecture is related to the question whether for a
torsionfree group $G$ the group ring $\IQ G$ can be embedded into a skew field, see for
instance~\cite{Henneke-Kielak(2020)}.

There is a program of Linnell~\cite{Linnell(1993)} to prove the Atyiah Conjecture which is
discussed in details for instance in~\cite[Theorem~10.38 on page~387 and
Section~10.3]{Lueck(2002)} and~\cite{Reich(2006)}.  This shows that one has at least some
ideas why the Atyiah Conjecture is true and that the Atiyah
Conjecture is related to some deep ring theory and to
algebraic $K$-theory, notably to projective class groups.  This connection to ring theory
has been explained and exploited for instance in~\cite{Jaikin-Zapirain(2019), Jaikin-Zapirain+Lopez-Alvarez(2020)},
where the division closure is replaced by  the $\ast$-regular closure.

For more information about the Atyiah Conjecture we refer for
instance to~\cite[Chapter~10]{Lueck(2002)}.


\subsection{The Singer Conjecture}%
\label{subsec:The_Singer_Conjecture}

\begin{conjecture}[Singer Conjecture]\label{con:Singer_Conjecture}
  If $M$ is an aspherical closed manifold, then
  \[
    b_p^{(2)}(\widetilde{M}) = 0 \quad \text{if} \; 2p \not= \dim(M).
  \]
  If $M$ is a closed connected Riemannian manifold with negative sectional curvature of
  even dimension $\dim (M) = 2m$, then
  \[
    (-1)^m \cdot \chi(M) = b_m^{(2)}(\widetilde{M}) > 0.
  \]
\end{conjecture}

Note that the equality $(-1)^m \cdot \chi(M) = b_m^{(2)}(\widetilde{M})$ appearing in the
Singer Conjecture~\ref{con:Singer_Conjecture} above follows from the the Euler-Poincar\'e
formula $\chi(M) = \sum_{p \ge 0} (-1)^p \cdot b_p^{(2)}(\widetilde{M})$.  Obviously
Singer Conjecture~\ref{con:Singer_Conjecture} implies the following conjecture in the
cases, where $M$ is aspherical or has negative sectional curvature.

\begin{conjecture}[Hopf Conjecture]\label{con:Hopf_Conjecture}
  If $M$ is an aspherical closed manifold of even dimension $\dim (M) = 2m$, then
  \[
    (-1)^{m} \cdot \chi(M) \ge 0.
  \]
  If $M$ is a closed Riemannian manifold of even dimension $\dim (M) = 2m$ with sectional
  curvature $\sec(M)$, then
  \[
    \begin{array}{rlllllll}
      (-1)^m \cdot \chi(M) & > & 0 & &
                                       \text{if}   & \sec(M) & < & 0;
      \\
      (-1)^m \cdot \chi(M) & \ge   & 0 & &
                                           \text{if}   & \sec(M) & \le & 0;
      \\
      \chi(M) & = & 0 & &
                          \text{if}  & \sec(M) & = & 0;
      \\
      \chi(M) & \ge & 0 & &
                            \text{if}   & \sec(M) & \ge & 0;
      \\
      \chi(M) & > & 0 & &
                          \text{if}  & \sec(M) & > & 0.
    \end{array}
  \]
\end{conjecture}

In original versions of the Singer Conjecture~\ref{con:Singer_Conjecture} and the Hopf
Conjecture~\ref{con:Hopf_Conjecture} the condition aspherical closed manifolds was
replaced by the condition closed Riemannian manifold with non-positive sectional
curvature. Note that a closed Riemannian manifold with non-positive sectional curvature is
aspherical by Hadamard's Theorem.

Note that the Singer Conjecture~\ref{con:Hopf_Conjecture} is consistent with the Atiyah
Conjecture in the sense that it predicts that the $L^2$-Betti numbers
$b_p^{(2)}(\widetilde{M})$ for an aspherical closed manifold $M$ are all integers.

The \emph{action dimension} of a discrete group $G$ is the smallest dimension of a
contractible manifold that admits a proper action of $G$.  This notion and its relation to
the Singer Conjecture is explained
in~\cite{Avramidi-Davis-Okun-Schreve(2016),Okun-Schreve(2016)}.

In contrast to the Atyiah Conjecture the evidence for the
Singer Conjecture~\ref{con:Singer_Conjecture} comes from computations only and no good
strategy is known for a potential proof.  In some sense Poincar\'e duality and the
$L^2$-conditions seems to force the $L^2$-Betti numbers $b_p^{(2)}(\widetilde{M})$ of an
aspherical closed manifold to concentrate in the middle dimension.  One may wonder what
happens if we replace $M$ by an aspherical finite Poincar\'e complex in the Singer
Conjecture~\ref{con:Singer_Conjecture}.  There are counterexamples to the Singer
Conjecture~\ref{con:Singer_Conjecture} if one weakens aspherical to rationally aspherical,
see~\cite[Theorem~4]{Avramidi(2018)}.

For more information about the Singer Conjecture and its status we refer for instance
to~\cite[Conjecture 2]{Dodziuk(1979)},~\cite[Chapter~11]{Lueck(2002)},
and~\cite{Singer(1977)}.


\subsection{The Determinant Conjecture}%
\label{subsec:The-Determinant_Conjecture}

\begin{conjecture}[Determinant Conjecture for a group $G$]\label{con:Determinant_Conjecture}
  For any matrix $A \in M_{r,s}(\IZ G)$, the Fuglede-Kadison determinant of the morphism
  of Hilbert modules $r_A^{(2)}\colon L^2(G)^r \to L^2(G)^s$ given by right multiplication
  with $A$ satisfies
  \[{\det}_{\caln(G)}^{(2)}\bigl(r_A^{(2)}\bigr)\ge 1.
  \]
\end{conjecture}

\begin{remark}[Status of the Determinant Conjecture]\label{rem:status_of_Determinant_Conjecture}
  We will want to assume that the Determinant Conjecture~\ref{con:Determinant_Conjecture}
  is true because then the condition of determinant class is automatically satisfied. This
  is an acceptable condition since the Determinant
  Conjecture~\ref{con:Determinant_Conjecture} is known for a very large class of groups.
  Namely, the following is known, see~\cite[Theorem~5]{Elek-Szabo(2005)},~%
\cite[Section~13.2]{Lueck(2002)},~\cite[Theorem~1.21]{Schick(2001b)}.  Let $\calf$ be
  the class of groups for which the Determinant
  Conjecture~\ref{con:Determinant_Conjecture} is true.  Then:
  \begin{enumerate}

  \item\label{rem:status_of_Determinant_Conjecture:amenable_quotient}
    Amenable quotient\\
    Let $H \subset G$ be a normal subgroup. Suppose that $H \in \calf$ and the quotient
    $G/H$ is amenable. Then $G \in \calf$;

  \item\label{rem:status_of_Determinant_Conjecture:direct_limit}
    Colimits\\
    If $G = \colim_{i \in I} G_i$ is the colimit of the directed system
    $\{G_i \mid i \in I\}$ of groups indexed by the directed set $I$ (with not necessarily
    injective structure maps) and each $G_i$ belongs to $\calf$, then $G$ belongs to
    $\calf$;

  \item\label{rem:status_of_Determinant_Conjecture:inverse_limit}
    Inverse limits\\
    If $G = \lim_{i \in I} G_i$ is the limit of the inverse system $\{G_i \mid i \in I\}$
    of groups indexed by the directed set $I$ and each $G_i$ belongs to $\calf$, then $G$
    belongs to $\calf$;

  \item\label{rem:status_of_Determinant_Conjecture:subgroups}
    Subgroups\\
    If $H$ is isomorphic to a subgroup of a group $G$ with $G \in \calf$, then
    $H \in \calf$;

  \item\label{rem:status_of_Determinant_Conjecture:quotient_with_finite_kernel}
    Quotients with finite kernel\\
    Let $1 \to K \to G \to Q \to 1$ be an exact sequence of groups. If $K$ is finite and
    $G$ belongs to $\calf$, then $Q$ belongs to $\calf$;

  \item\label{rem:status_of_Determinant_Conjecture:sofic_groups} Sofic groups belong to
    $\calf$;

  \item\label{rem:status_of_Determinant_Conjecture:3-manifold_groups} The fundamental
    group of a $3$-manifold belongs to $\calf$.
  \end{enumerate}

  The class of sofic groups is very large.  It is closed under direct and free products,
  taking subgroups, taking inverse and direct limits over directed index sets, and is
  closed under extensions with amenable groups as quotients and a sofic group as kernel.
  In particular it contains all residually amenable groups and fundamental groups of
  $3$-manifolds.  One expects that there exists non-sofic groups but no example is known.
  More information about sofic groups can be found for instance in~\cite{Elek-Szabo(2006)}
  and~\cite{Pestov(2008)}.

\end{remark}

For more information about the Determinant Conjecture we refer for instance
to~\cite[Chapter~13]{Lueck(2002)}.


\subsection{Approximation Conjecture for $L^2$-Betti numbers}%
\label{subsec:The-Approximation_Conjecturs_for_L2-Betti_numbers}

Let $G$ be a group together with an \emph{exhausting normal inverse system of subgroups}
$\{G_i \mid i \in I\}$ of normal subgroups of $G$ directed by inclusion over the directed
set $I$ such that $\bigcap_{i \in I} G_i = \{1\}$.  If $I$ is given by the natural
numbers, this boils down to a nested sequence of normal subgroups of $G$
\[
  G = G_0 \supset G_1 \supseteq G_2 \supseteq \cdots
\]
satisfying $\bigcap_{n \ge 1} G_n = \{1\}$.

\begin{notation}[Inverse systems and matrices]\label{not:inverse_systems_and_matrices}
  Let $R$ be a ring with $\IZ \subseteq R \subseteq \IC$. Given a matrix
  $A \in M_{r,s}(RG)$, let $A[i]\in M_{r,s}(R[G/G_i])$ be the matrix obtained from $A$ by
  applying elementwise the ring homomorphism $RG \to R[G/G_i]$ induced by the projection
  $G \to G/G_i$.  Let $r_A \colon RG^r \to RG^s$ and
  $r_{A[i]} \colon R[G/G_i]^r \to R[G/G_i]^s$ be the $RG$- and $R[G/G_i]$-homomorphisms
  given by right multiplication with $A$ and $A[i]$.  Let
  $r_A^{(2)} \colon L^2(G)^r \to L^2(G)^s$ and
  $r_{A[i]}^{(2)} \colon L^2(G/G_i)^r \to L^2(G/G_i)^s$ be the morphisms of Hilbert
  $\caln(G)$- and Hilbert $\caln(G/G_i)$-modules given by right multiplication with $A$
  and $A[i]$.
\end{notation}

\begin{conjecture}[Approximation Conjecture for $L^2$-Betti numbers]%
\label{con:Approximation_conjecture_for_L2-Betti_numbers}
  A group $G$ together with an exhausting normal inverse system of subgroups
  $\{G_i \mid i \in I\}$ satisfies the \emph{Approximation Conjecture for $L^2$-Betti
    numbers} if one of the following equivalent conditions holds:

  \begin{enumerate}

  \item Matrix version\\[1mm]
    Let $A \in M_{r,s}(\IQ G)$ be a matrix. Then
    \begin{eqnarray*}
      \lefteqn{\dim_{\caln(G)}\bigl(\ker\bigl(r_A^{(2)}\colon L^2(G)^r \to L^2(G)^s
      \bigr)\bigr)}
      & &
      \\ & \hspace{14mm} =  &
                              \lim_{i \in I} \;\dim_{\caln(G/G_i)}\big(\ker
                              \big(r_{A[i]}^{(2)}\colon L^2(G/G_i)^r \to L^2(G/G_i)^s \bigr)\bigr);
    \end{eqnarray*}

  \item $CW$-complex version\\[1mm]
    Consider normal covering $p \colon \overline{X} \xrightarrow{p} X$ with $G$ as group
    of deck transformation over a $CW$-complex $X$ of finite type. Put
    $\overline{X}[i] := G_i\backslash \overline{X}$.  Then we get a normal covering
    $p[i] \colon \overline{X}[i] \to X$ with $G/G_i$ as group of deck transformation and
    \begin{eqnarray*}
      b_p^{(2)}(\overline{X};\caln(G)) & = & \lim_{i \in I} \; b_p^{(2)}(\overline{X}[i];\caln(G/G_i)).
    \end{eqnarray*}

  \end{enumerate}
\end{conjecture}

The two conditions appearing in
Conjecture~\ref{con:Approximation_conjecture_for_L2-Betti_numbers} are equivalent
by~\cite[Lemma~13.4 on page~455]{Lueck(2002)}.

\begin{theorem}[The Determinant Conjecture implies the Approximation Conjecture for $L^2$-Betti numbers]%
\label{the:The_Determinant_Conjecture_implies_the_Approximation_Conjecture_for_L2-Betti_numbers}
  If for each $i \in I$ the quotient $G/G_i$ satisfies the Determinant
  Conjecture~\ref{con:Determinant_Conjecture}, then the conclusion of the Approximation
  Conjecture~\ref{con:Approximation_conjecture_for_L2-Betti_numbers} holds for
  $\{G_i \mid i \in I\}$.
\end{theorem}
\begin{proof}
  See~\cite[Theorem~13.3 (1) on page~454]{Lueck(2002)} and~\cite{Schick(2001b)}.
\end{proof}

Suppose that each quotient $G/G_i$ is finite. Then we
rediscover~\eqref{list:main_properties_of_rho2(widetildeX):approximation} appearing in
Subsection~\ref{subsec:Basic_properties_of_L2-Betti_numbers_and_L2-torsion_of_universal_coverings_of_finite_CW-complexes}
from Remark~\ref{rem:status_of_Determinant_Conjecture} and
Theorem~\ref{the:The_Determinant_Conjecture_implies_the_Approximation_Conjecture_for_L2-Betti_numbers}.

For more information about the Approximation Conjecture for $L^2$-Betti
numbers~\ref{con:Approximation_conjecture_for_L2-Betti_numbers} we refer for instance
to~\cite[~Chapter~13]{Lueck(2002)} and~\cite[Conjecture~1.10]{Schick(2001b)}.


\subsection{Approximation Conjectures for Fuglede-Kadison determinants and $L^2$-torsion}%
\label{subsec:The_Approximation_Conjecture_for_Fuglede-Kadison_determinants_and_L2-torsion}

Next we turn to Fuglede-Kadison determinants and $L^2$-torsion.


\subsubsection{Approximation Conjecture for Fuglede-Kadison determinants}%
\label{subsubsec:Approximation_Conjecture_for_Fuglede-Kadison_determinants}

\begin{conjecture}[Approximation Conjecture for Fuglede-Kadison determinants]%
\label{con:Approximation_conjecture_for_Fuglede-Kadison_determinants_with_arbitrary_index}
  A group $G$ together with an exhausting normal inverse system of subgroups
  $\{G_i \mid i \in I\}$ satisfies the \emph{Approximation Conjecture for Fuglede-Kadison
    determinants} if for any matrix $A \in M_{r,s}(\IQ G)$ we get for the Fuglede-Kadison
  determinant
  \begin{eqnarray*}
    {\det}_{\caln(G)}\bigl(r_A^{(2)}\colon L^2(G)^r \to L^2(G)^s\bigr)   & > &  0;
    \\
    {\det}_{\caln(G/G_i)}\big(r_{A[i]}^{(2)}\colon L^2(G/G_i)^r \to L^2(G/G_i)^s\bigr)   & > &  0,
  \end{eqnarray*}
  and
  \begin{eqnarray*}
    \lefteqn{{\det}_{\caln(G)}\bigl(r_A^{(2)}\colon L^2(G)^r \to L^2(G)^s\bigr)}
    & &
    \\ & \hspace{14mm} =  &
                            \lim_{i \in I}\; {\det}_{\caln(G/G_i)}\big(r_{A[i]}^{(2)}\colon L^2(G/G_i)^r \to L^2(G/G_i)^s\bigr),
  \end{eqnarray*}
  where the existence of the limit above is part of the claim.
\end{conjecture}

\begin{remark}[$\IQ$-coefficients are necessary]\label{rem:IQ-coefficients_are_necessary}
  Recall that the Atiyah Conjecture may be true if we consider
  matrices over the complex group ring instead of the rational group ring.
  Conjecture~\ref{con:Approximation_conjecture_for_Fuglede-Kadison_determinants_with_arbitrary_index}
  does not hold if one replaces $\IQ$ by $\IC$ by the following result appearing
  in~\cite[Example~13.69 on page~481]{Lueck(2002)}.

  There exists a sequence of integers $2 \le n_1 < n_2 < n_3 < \cdots$ and a real number
  $s$ such that for $G = \IZ$ and $G_i = n_i \cdot \IZ$ and the $(1,1)$-matrix $A$ given
  by the element $z-\exp(2\pi i s)$ in $\IC[\IZ] = \IC[z,z^{-1}]$ we get for all $i \ge 1$
  \begin{eqnarray*}
    \ln\bigl({\det}^{(2)}_{\caln(G)}(r_A^{(2)})\bigr) & = & 0;
    \\
    \frac{\ln\bigl(\det(r_{A[i]}^{(2)})\bigr)}{[G:G_i]} & \le & - 1/2.
  \end{eqnarray*}
\end{remark}

A strategy for the proof of
Conjecture~\ref{con:Approximation_conjecture_for_Fuglede-Kadison_determinants_with_arbitrary_index}
is discussed in~\cite[Section~17]{Lueck(2016_l2approx)}, see also~\cite{Koch-Lueck(2014)}.
The uniform integrability condition appearing in~\cite[Theorem~16.3~(v) and
Remark~16.13]{Lueck(2016_l2approx)} seems to play a key role. It would be automatically
satisfied if one has a uniform estimate on the spectral density functions of the
intermediate stages for $i \in I$. Roughly speaking, the spectrum has to be uniformly thin
at zero for $r_{A[i]}^{(2)}$ each $i \in I$.  The crudest way to guarantee this condition
is to require a uniform gap at zero in the spectrum for $r_{A[i]}^{(2)}$ each $i \in I$,
see~\cite[Lemma~16.14 and Remark~16.15]{Lueck(2016_l2approx)}.


\subsubsection{The chain complex version}%
\label{subsubsec:The_chain_complex_version}

\begin{notation}[Inverse systems and chain complexes]\label{not:inverse_system_and_chain_complexes}
  Let $C_*$ be a finite based free $\IQ G$-chain complex.  In the sequel we denote by
  $C[i]_*$ the $\IQ [G/G_i]$-chain complex $\IQ [G/G_i]\otimes_{\IQ G} C_* $, by
  $C^{(2)}_*$ the finite Hilbert $\caln(G)$-chain complex $L^2(G) \otimes_{\IQ G} C_*$,
  and by $C[i]^{(2)}_*$ the finite Hilbert $\caln(G/G_i)$-chain complex
  $L^2(G/G_i) \otimes_{\IQ[G/G_i]} C[i]_* $. The $\IQ G$-basis for $C_*$ induces a
  $\IQ[G/G_i]$-basis for $C[i]_*$ and Hilbert space structures on $C^{(2)}_*$ and
  $C[i]_*^{(2)}$ using the standard Hilbert structure on $L^2(G)$ and $L^2(G/G_i)$.  We
  emphasize that in the sequel after fixing a $\IQ G$-basis for $C_*$ the
  $\IQ[G/G_i]$-basis for $C_*[i]$ and the Hilbert structures on $C_*^{(2)}$ and
  $C[i]_*^{(2)}$ have to be chosen in this particular way.

  Denote by
  \begin{eqnarray}
    \rho^{(2)}(C_*) 
    & := & 
           - \sum_{p \ge 0} (-1)^p \cdot \ln\bigl({\det}_{\caln(G)}(c_p^{(2)})\bigr);
           \label{L2-torsion_for_C_over_cakln(G)}
    \\
    \rho^{(2)}(C[i]_*) 
    & := & 
           - \sum_{p \ge 0} (-1)^p \cdot \ln\bigl({\det}_{\caln(G/G_i)}(c[i]_p^{(2)})\bigr),
           \label{L2-torsion_for_C[i]_over_caln(G/G_i)}
  \end{eqnarray}
  their \emph{$L^2$-torsion} over $\caln(G)$ and $\caln(G/G_i)$ respectively, provided
  that $C_*$ and $C[i]_*$ are of determinant class.
\end{notation}

We have the following chain complex version of
Conjecture~\ref{con:Approximation_conjecture_for_Fuglede-Kadison_determinants_with_arbitrary_index}
which is obviously equivalent to
Conjecture~\ref{con:Approximation_conjecture_for_Fuglede-Kadison_determinants_with_arbitrary_index}

\begin{conjecture}[Approximation Conjecture for $L^2$-torsion of chain complexes]%
\label{con:Approximation_conjecture_for_L2-torsion_of_chain_complexes}
  A group $G$ together with an exhausting normal inverse system $\{G_i \mid i \in I\}$
  satisfies the \emph{Approximation Conjecture for $L^2$-torsion of chain complexes} if
  the finite based free $\IQ G$-chain complex $C_*$ and $C[i]_*$ are of determinant class
  and we have
  \[
    \rho^{(2)}(C_*) = \lim_{i \in I} \;\rho^{(2)}(C[i]_*).
  \]
\end{conjecture}


\subsubsection{Analytic $L^2$-torsion}%
\label{subsubsec:Analytic_L2-torsion}

Let $\overline{M}$ be a Riemannian manifold without boundary that comes with a proper free
cocompact isometric $G$-action.  Denote by $\overline{M}[i]$ the Riemannian manifold
obtained from $\overline{M}$ by dividing out the $G_i$-action. The Riemannian metric on
$\overline{M}[i]$ is induced by the one on $M$. There is an obvious proper free cocompact
isometric $G/G_i$-action on $\overline{M}[i]$ induced by the given $G$-action on
$\overline{M}$. Note that $M = \overline{M}/G$ is a closed Riemannian manifold and we get
a $G$-covering $\overline{M} \to M$ and a $G/G_i$-covering $\overline{M}[i] \to M$ for the
closed Riemannian manifold $M= \overline{M}/G$, which are compatible with the Riemannian
metrics. Denote by
\begin{eqnarray}
  &\rho^{(2)}_{\an}(\overline{M};\caln(G))  \in \IR;&
                                                      \label{L2-torsion_for_M_over_cakln(G)}
  \\
  &\rho^{(2)}_{\an}(\overline{M}[i];\caln(G/G_i))  \in  \IR,&
                                                              \label{L2-torsion_for_M[i]_over_caln(G/G_i)}
\end{eqnarray}
their \emph{analytic $L^2$-torsion} over $\caln(G)$ and $\caln(G/G_i)$ respectively,
provided that $\overline{M}$ and $\overline{M}[i]$ are of determinant class.  For the
notion of analytic $L^2$-torsion we refer for instance to~\cite[Chapter~3]{Lueck(2002)}.
Burghelea-Friedlander-Kappeler-McDonald~\cite{Burghelea-Friedlander-Kappeler-McDonald(1996a)}
have shown that the analytic $L^2$-torsion agrees with the $L^2$-torsion defined in terms
of the cellular chain complex in~\eqref{rho_upper_(2)(overline(X))} in the $L^2$-acyclic
case.

We will not discuss the condition of determinant class here and in the sequel.  This is
not necessary if each $G_i$ satisfies the Determinant
Conjecture~\ref{con:Determinant_Conjecture}, which is true for a very large class of
groups, see Remark~\ref{rem:status_of_Determinant_Conjecture}.

\begin{conjecture}[Approximation Conjecture for analytic $L^2$-torsion]%
\label{con:Approximation_conjecture_for_analytic_L2-torsion}
  Consider a group $G$ together with an exhausting normal inverse system
  $\{G_i \mid i \in I\}$.  Let $\overline{M}$ be a Riemannian manifold without boundary
  that comes with a proper free cocompact isometric $G$-action. Then $\overline{M}$ and
  $\overline{M}[i]$ are of determinant class and
  \[\rho^{(2)}_{\an}(\overline{M};\caln(G))
    = \lim_{i \in I} \;\rho^{(2)}_{\an}(M[i];\caln(G/G_i)).
  \]
\end{conjecture}

\begin{remark}\label{rem:main_theorem_answers_questions}
  The conjectures above imply a positive answer
  to~\cite[Question~21]{Deninger(2009Mahler)} and~\cite[Question~13.52 on page~478 and
  Question~13.73 on page~483]{Lueck(2002)}.  They also would settle~\cite[Problem~4.4 and
  Problem~6.4]{Kitano-Morifuji(2008)} and~\cite[Conjecture~3.5]
  {Kitano-Morifuji-Takasawa(2004surface_bundle)}. One may wonder whether it is related to
  the Volume Conjecture due to Kashaev~\cite{Kashaev(1997)} and H. and
  J. Murakami~\cite[Conjecture~5.1 on page~102]{Murakami-Murakami(2001)}.

\end{remark}

The proof of the following result can be found in~\cite[Section~16]{Lueck(2016_l2approx)}.
It reduces in the weakly acyclic case
Conjecture~\ref{con:Approximation_conjecture_for_analytic_L2-torsion} to
Conjecture~\ref{con:Approximation_conjecture_for_Fuglede-Kadison_determinants_with_arbitrary_index}.

\begin{theorem}\label{the:comparing_analytic_and_chain_complexes}
  Consider a group $G$ together with an exhausting normal inverse system
  $\{G_i \mid i \in I\}$.  Let $\overline{M}$ be a Riemannian manifold without boundary
  that comes with a proper free cocompact isometric $G$-action. Suppose that
  $b_p^{(2)}(\overline{M};\caln(G)) = 0$ for all $p \ge 0$. Assume that the Approximation
  Conjecture for $L^2$-torsion of chain
  complexes~\ref{con:Approximation_conjecture_for_L2-torsion_of_chain_complexes} (or,
  equivalently,
  Conjecture~\ref{con:Approximation_conjecture_for_Fuglede-Kadison_determinants_with_arbitrary_index})
  holds for $G$.

  Then Conjecture~\ref{con:Approximation_conjecture_for_analytic_L2-torsion} holds for
  $M$, i.e., $\overline{M}$ and $\overline{M}[i]$ are of determinant class and
  \[
    \rho^{(2)}_{\an}(\overline{M};\caln(G)) = \lim_{i \in I}
    \;\rho^{(2)}_{\an}(\overline{M}[i];\caln(G/G_i)).
  \]
\end{theorem}

Note that in Theorem~\ref{the:comparing_analytic_and_chain_complexes} we are not assuming
that $b_p^{(2)}(\overline{M}[i];\caln(G/G_i))$ vanishes for all $p \ge 0$ and $i \in I$.

It is conceivable that Theorem~\ref{the:comparing_analytic_and_chain_complexes} remains 
true if we drop the assumption that $b_p^{(2)}(\overline{M};\caln(G))$ vanishes for all
$p \ge 0$, but our present proof works only under this assumption,
see~\cite[Remark~16.2]{Lueck(2016_l2approx)}.

More information about the conjectures above can be found in~\cite[Section~15 --
17]{Lueck(2016_l2approx)}.


\subsection{Homological growth and $L^2$-torsion}%
\label{subsec:Homological_growth_and_L2_torsion}

Denote by $H_n(X;\IZ)$ the singular homology with integer coefficients.  If $X$ is a
compact manifold or a finite $CW$-complex, then $H_n(X;\IZ)$ is a finitely generated group
and hence its torsion part $\tors(H_n(X;\IZ))$ is a finite abelian group.

\begin{conjecture}[Homological growth and $L^2$-torsion for aspherical manifolds]%
\label{con:Homological_growth_and_L2-torsion_for_aspherical_manifolds}
  Let $M$ be an aspherical closed manifold of dimension $d$ with fundamental group
  $G = \pi_1(M)$.  Consider a nested sequence
  $G = G_0 \supseteq G_1 \supseteq G_2 \supseteq \cdots $ of normal subgroups of $G$ of
  finite index $[G:G_i]$ satisfying $\bigcap_{i = 0}^{\infty} G_i = \{1\}$. Let
  $M[i] = \widetilde{M} /G_i \to M$ be the $[G:G_i]$-sheeted covering of $M$ associated to
  $G_i \subseteq G$.
 
  Then we get for any natural number $n$ with $2n +1 \not= d$
  \[
    \lim_{i \to \infty} \;\frac{\ln\big(\bigl|\tors(H_n(M[i];\IZ))\bigr|\bigr)}{[G:G_i]} =
    0,
  \]
  and we get in the case $d = 2n+1$
  \[
    \lim_{i \to \infty} \;\frac{\ln\big(\bigl|\tors(H_n(M[i];\IZ))\bigr|\bigr)}{[G:G_i]} =
    (-1)^n \cdot \rho^{(2)}_{\an}(\widetilde{M}) \ge 0.
  \]
\end{conjecture}

Recall that $\rho^{(2)}_{\an}(\widetilde{M}) = \rho^{(2)}(\widetilde{M})$ holds, if
$\widetilde{M}$ is $L^2$-acyclic and that the Singer Conjecture implies for an aspherical
closed manifold of odd dimension that $\widetilde{M}$ is $L^2$-acyclic.  Moreover, since
$G$ appearing in
Conjecture~\ref{con:Homological_growth_and_L2-torsion_for_aspherical_manifolds} is
residually finite, the condition of determinant class is automatically satisfied for
$\widetilde{M}$.

One may wonder what happens if we replace $M$ by an aspherical finite Poincar\'e complex
in Conjecture~\ref{con:Homological_growth_and_L2-torsion_for_aspherical_manifolds}.

Conjecture~\ref{con:Homological_growth_and_L2-torsion_for_aspherical_manifolds} is known
to be true in the case that $G$ contains a normal infinite elementary amenable subgroup or admits 
a non-trivial $S^1$-action, see~\cite{Lueck(2013l2approxfib)}. However, to the author's
knowledge there is no hyperbolic $3$-manifold for which
Conjecture~\ref{con:Homological_growth_and_L2-torsion_for_aspherical_manifolds} is known
to be true.

Conjecture~\ref{con:Homological_growth_and_L2-torsion_for_aspherical_manifolds} is
attributed to Bergeron-Venkatesh~\cite{Bergeron-Venkatesh(2013)}. They allow only locally
symmetric spaces for $M$.  They also consider the case of  twisting with a
finite-dimensional integral representation.  Further discussions about this conjecture can
be found for instance
in~\cite[Section~7.5.1]{Aschenbrenner-Friedl-Wilton(2015)},~\cite{Bergeron-Sengun-Venkatesh(2016)},
and~\cite{Brock-Dunfield(2015)}.

The relation between
Conjecture~\ref{con:Approximation_conjecture_for_Fuglede-Kadison_determinants_with_arbitrary_index}
and Conjecture~\ref{con:Homological_growth_and_L2-torsion_for_aspherical_manifolds} is
discussed in~\cite[Section~9 and~10]{Lueck(2016_l2approx)}.

The chain complex version
Conjecture~\ref{con:Homological_growth_and_L2-torsion_for_aspherical_manifolds} is stated
in~\cite[Conjeture~7.12]{Lueck(2016_l2approx)}. We at least explain what it says for
$1$-dimensional chain complexes, or, equivalently, matrices. Here it is important to work
over the integral group ring.

\begin{conjecture}[Approximating Fuglede-Kadison determinants by homology]%
\label{con:Approximating_Fuglede-Kadison_determinants_by_homology}\ Consider a nested
  sequence $G = G_0 \supseteq G_1 \supseteq G_2 \supseteq \cdots $ of normal subgroups of
  $G$ of finite index $[G:G_i]$ satisfying $\bigcap_{i = 0}^{\infty} G_i = \{1\}$.
  Consider $A \in M_{r,r}(\IZ G)$.

  Then we get using Notation~\ref{not:inverse_systems_and_matrices} for $R = \IZ$
  \begin{eqnarray*}
    {\det}_{\caln(G)}^{(2)}(r_A^{(2)})
    & = & 
          \lim_{i \to \infty} \; \bigl|\tors(\coker(r_{A[i]}))\bigr|^{1/[G:G_i]}.
  \end{eqnarray*}
\end{conjecture}
  
Recall that for
Conjecture~\ref{con:Approximation_conjecture_for_Fuglede-Kadison_determinants_with_arbitrary_index}
we could formulate a good condition, namely the uniform integrability condition,
which implies its validity.  Nothing like this is known for
Conjecture~\ref{con:Approximating_Fuglede-Kadison_determinants_by_homology}.  The only
infinite group for which
Conjecture~\ref{con:Approximating_Fuglede-Kadison_determinants_by_homology} is known to be
true is $\IZ$. The proof indicates that some deep number theory may enter in a potential
proof of Conjecture~\ref{con:Approximating_Fuglede-Kadison_determinants_by_homology}.
Note that for Conjecture~\ref{con:Approximating_Fuglede-Kadison_determinants_by_homology}
it is crucial that the matrix $A$ lives over the integral group ring, whereas for
Conjecture~\ref{con:Approximation_conjecture_for_Fuglede-Kadison_determinants_with_arbitrary_index}
it suffices that $A$ lives over the rational group ring, see
Remark~\ref{rem:IQ-coefficients_are_necessary}.


\subsection{$L^2$-invariants and the simplicial volume}%
\label{subsec:L2-invariant_and_the_simplicial_volume}

We briefly recall the definition of the \emph{simplicial volume}.

Let $X$ be a topological space and let $C_*^{\sing}(X;\IR)$ be its singular chain complex
with real coefficients. Recall that a singular $p$-simplex of $X$ is a continuous map
$\sigma\colon \Delta_p \to X$, where here $\Delta_p$ denotes the standard $p$-simplex (and
not the Laplace operator). Let $S_p(X)$ be the set of all singular $p$-simplices.  Then
$C_p^{\sing}(X;\IR)$ is the real vector space with $S_p(X)$ as basis. The $p$-th
differential $\partial_p$ sends the element $\sigma $ given by a $p$-simplex
$\sigma\colon \Delta_p \to X$ to $\sum_{i = 0}^p (-1)^{i} \cdot \sigma \circ s_i$, where
$s_i\colon \Delta_{p-1} \to \Delta_p$ is the $i$-th face map. Define the
\emph{$L^1$-norm}%
\index{L1-norm of a singular chain@$L^1$-norm of a singular chain} of an element
$x \in C_p^{\sing}(X;\IR)$, which is given by the (finite) sum
$\sum_{\sigma \in S_p(X)} \lambda_{\sigma} \cdot \sigma$, by
\begin{eqnarray*}
  ||x||_1 & := & \sum_{\sigma} |\lambda_{\sigma}|.
\end{eqnarray*}

We define the \emph{$L^1$-seminorm}%
\index{L1-seminorm of a homology class@$L^1$-seminorm of a homology class} of an element
$y$ in the $p$-th singular homology $H_p^{\sing}(X;\IR) := H_p(C_*^{\sing}(X;\IR))$ by
\begin{eqnarray*}
  ||y||_1
  & := & \inf\{||x||_1 \mid
         x \in C_p^{\sing}(X;\IR), \partial_p(x) = 0, y = [x]\}.
\end{eqnarray*}
Notice that $||y||_1$ defines only a seminorm on $H_p^{\sing}(X;\IR)$, it is possible that
$||y||_1 = 0$ but $y \not= 0$. The next definition is taken
from~\cite[page~8]{Gromov(1982)}.

\begin{definition}[Simplicial volume]\label{def:simplicial_volume_of_an_orientable_closed_manifold}
  Let $M$ be a closed connected orientable manifold of dimension $n$. Define its
  \emph{simplicial volume} to be the non-negative real number
  \[
    ||M|| := ||j([M])||_1 \; \in \IR^{\ge 0}
  \]
  for any choice of fundamental class $[M] \in H_n^{\sing}(M;\IZ)$ and
  $j\colon H_n^{\sing}(M;\IZ) \to H_n^{\sing}(M;\IR)$ the change of coefficients map
  associated to the inclusion $\IZ \to \IR$.
\end{definition}

There is the following interesting but poorly understood conjecture relating the
simplicial volume and $L^2$-invariants for aspherical orientable closed manifolds,
see~\cite[Chapter~14.1]{Lueck(2002)}.

\begin{conjecture}[Simplicial volume and $L^2$-invariants]\label{con:simplicial_volume_and_L2-invariants}
  Let $M$ be an aspherical closed orientable manifold of dimension $\ge 1$.  Suppose that
  its simplicial volume $||M||$ vanishes. Then $\widetilde{M}$ is of determinant class and
  \begin{eqnarray*}
    b_p^{(2)}(\widetilde{M}) & = & 0 \hspace{5mm} \mbox{ for } p \ge 0;
    \\
    \rho^{(2)}(\widetilde{M}) & = & 0.
  \end{eqnarray*}
\end{conjecture}

For more information about this conjecture we refer for instance
to~\cite{Frigerio-Loeh-Pagliantini_Sauer(2016)},~\cite[Chapter~14]{Lueck(2002)},~\cite{Schmidt(2005)}.
It has been verified by computations if $M$ is a locally symmetric space, if $M$ is
$3$-manifold, if $M$ carries a non-trivial $S^1$-action, or $\pi_1(M)$ is elementary
amenable.  But no strategy for a potential proof is known to the author.


\typeout{---- Section 4: The computation of $L^2$-Betti numbers and $L^2$-torsion of
  $3$-manifolds -----}

\section{The computation of $L^2$-Betti numbers and $L^2$-torsion of $3$-manifolds}%
\label{sec:The_computation_of_L2-Betti_numbers_and_L2-torsion_of_3-manifolds}


\subsection{$L^2$-Betti numbers of $3$-manifolds}%
\label{subsec:L2-Betti_numbers_of_3-manifolds}

The following theorem is taken from~\cite[Theorem~0.1]{Lott-Lueck(1995)}.

\begin{theorem}[$L^2$-Betti numbers of $3$-manifolds]\label{the:L2-Betti_numbers_of_3-manifolds}
  Let $M$ be the connected sum \mbox{$M_1 \# \ldots \# M_r$} of compact connected
  orientable prime $3$-manifolds $M_j$.  Assume that $\pi_1(M)$ is infinite.  Then the
  $L^2$-Betti numbers of the universal covering $\widetilde{M}$ are given by
  \begin{eqnarray}
    b_0^{(2)}(\widetilde{M}) & = & 0; \nonumber
    \\
    b_1^{(2)}(\widetilde{M}) & = & (r-1) -
                                   \sum_{j=1}^r \frac{1}{\mid \pi_1(M_j)\mid}\; +
                                   \left|\{C \in \pi_0(\partial M) \mid C \cong S^2\}\right|  - \chi(M); \nonumber
    \\
    b_2^{(2)}(\widetilde{M}) & = & (r-1) -
                                   \sum_{j=1}^r \frac{1}{\mid \pi_1(M_j)\mid} \;+
                                   \left|\{C \in \pi_0(\partial M) \mid C \cong S^2\}\right|; \nonumber
    \\
    b_3^{(2)}(\widetilde{M}) & = & 0. \nonumber
  \end{eqnarray}
  In particular, $\widetilde{M}$ is $L^2$-acyclic if and only if $M$ is homotopy
  equivalent, or, equivalently, homeomorphic, to $\IR\IP^3 \# \IR\IP^3$ or a prime 3-manifold with infinite fundamental group whose
  boundary is empty or a union of tori.
\end{theorem}


\subsection{$L^2$-torsion of $3$-manifolds}%
\label{subsec:L2-Betti_torsion_of_3-manifolds}

Finally we state the values for the $L^2$-torsion, see~\cite[Theorem
0.6]{Lueck-Schick(1999)}.

\begin{theorem}[$L^2$-torsion of $3$-manifolds]\label{the:L2-torsion_of_irreducible_3-manifold}
  Let $M$ be a compact connected orientable prime $3$-manifold with infinite fundamental
  group such that the boundary of $M$ is empty or a disjoint union of incompressible tori.
  Let $M_1$, $M_2$, $\ldots$, $M_r$ be the hyperbolic pieces.  (They all have finite
  volume~\cite[Theorem B on page 52]{Morgan(1984)}.)

  Then $\widetilde{M}$ is $\det$-$L^2$-acyclic and
  \[
    \rho^{(2)}(\widetilde{M}) = -\frac{1}{6\pi}\cdot \sum_{i=1}^r \vol(M_i).
  \]
  In particular, $\rho^{(2)}(\widetilde{M})$ is $0$ if and and only if $M$ is
  $S^1 \times S^2$ or is a graph manifold.
\end{theorem}


\typeout{---- Section 5: The status of the conjectures about $L^2$-invariants for
  $3$-manifolds -----}

\section{The status of the conjectures about $L^2$-invariants for $3$-manifolds}%
\label{sec:The_status_of_the_conjectures_about_L2-invariants_for_3-manifolds}


\subsection{The Atiyah Conjecture}%
\label{subsec:The_Atiyah_Conjecture_3-manifolds}

The fundamental group of an admissible $3$-manifold $M$
that is not a closed graph manifold, is
torsionfree and satisfies the Atiyah Conjecture~\ref{con:Atiyah_Conjecture}, see
Theorem~\ref{the:Status_of_the_Atiyah_Conjecture}~\eqref{the:Status_of_the_Atiyah_Conjecture:3-manifold_not_graph}.


\subsection{The Singer Conjecture}%
\label{subsec:The_Singer_Conjecture_3-manifolds}

Every aspherical closed $3$-manifold satisfies the Singer
Conjecture~\ref{con:Singer_Conjecture} by
Theorem~\ref{the:L2-torsion_of_irreducible_3-manifold} since an aspherical closed
$3$-manifold is irreducible and has infinite fundamental group.


\subsection{The Determinant Conjecture}%
\label{subsec:The_Determinant_Conjecture_3-manifolds}

If $G$ is the fundamental group of a compact $3$-manifold, then $G$ is residually finite
and hence satisfies the Determinant Conjecture~\ref{con:Determinant_Conjecture} by
Remark~\ref{rem:status_of_Determinant_Conjecture}.


\subsection{Homological growth and $L^2$-torsion for aspherical manifolds}%
\label{subsec:Homological_growth_and_L2-torsion_for_aspherical_manifolds_3-manifolds}

The Conjecture~\ref{con:Homological_growth_and_L2-torsion_for_aspherical_manifolds} about
homological growth and $L^2$-torsion for aspherical manifolds, is wide open. To the
authors knowledge, there is no hyperbolic $3$-manifold, where it is known to be
true. Already this case would be very interesting.

Namely, suppose that $M$ is a closed hyperbolic $3$-manifold. Then
$\rho_{\an}(\widetilde{M})$ is known to be $- \frac{1}{6\pi} \cdot \vol(M)$, by
Theorem~\ref{the:L2-torsion_of_irreducible_3-manifold} and hence
Conjecture~\ref{con:Homological_growth_and_L2-torsion_for_aspherical_manifolds} predicts
\[
  \lim_{i \to \infty} \;\frac{\ln\big(\bigl|\tors(H_1(G_i))\bigr|\bigr)}{[G:G_i]} =
  \frac{1}{6\pi} \cdot \vol(M).
\]
Since the volume is always positive, the equation above implies that $|\tors(H_1(G_i))|$
grows exponentially in $[G:G_i]$.  Some evidence for
Conjecture~\ref{con:Homological_growth_and_L2-torsion_for_aspherical_manifolds} for closed
hyperbolic $3$-manifolds is given in Sun~\cite[Corollary~1,6]{Sun(2015)}, where it is
shown that for any finitely generated abelian group $A$, and any closed hyperbolic
3-manifold $M$, there exists a finite cover $N$ of $M$, such that $A$ is a direct summand
of $H_1(N;\IZ)$.

Bergeron-Sengun-Venkatesh~\cite{Bergeron-Sengun-Venkatesh(2016)} consider the equality
above for arithmetic hyperbolic $3$-manifolds and relate it to a conjecture about classes
in the second integral homology.

Some numerical evidence for the equality above is given in Sengun~\cite{Sengun(2011)}.

The inequality
\[
  \limsup _{i \to \infty} \;\frac{\ln\big(\bigl|\tors(H_1(G_i))\bigr|\bigr)}{[G:G_i]} \le
  \frac{1}{6\pi} \cdot \vol(M)
\]
is proved by Thang~\cite{Le(2018)} for a compact connected orientable irreducible
$3$-manifold $M$ with infinite fundamental group and empty or toroidal boundary.


\subsection{$L^2$-invariants and the simplicial volume}%
\label{subsec:L2-invariant_and_the_simplicial_volume_3-manifolds}

Define the positive real number $v_3$ to be the supremum of the volumes of all
$n$-dimensional geodesic simplices, i.e., the convex hull of $(n+1)$ points in general
position, in the $n$-dimensional hyperbolic space $\IH^3$.  If $M$ is an admissible
$3$-manifold, then one gets from~\cite[Theorem
0.6]{Lueck-Schick(1999)},~\cite{Soma(1981)}, and~\cite{Thurston(1978)},
see~\cite[Theorem~14.18 on page~490]{Lueck(2002)}
\[
  ||M|| = \frac{- 6\pi}{v_3} \cdot \rho^{(2)}(\widetilde{M}).
\]

In particular, $\rho^{(2)}(\widetilde{M}) = 0$ if and and only if $||M|| = 0$.  Hence
Conjecture~\ref{con:simplicial_volume_and_L2-invariants} is true in dimension~$3$.

It is not true for odd $n \ge 9$ that there exists a dimension constant $C_n$ such that
for an aspherical orientable closed manifold $M$ of dimension $n$ we have
$\rho^{(2)}(\widetilde{M}) = C_n \cdot ||M||$, see~\cite[Theorem~14.38 on
page~498]{Lueck(2002)}.

There are variants of the simplicial volume, namely, the notion of the \emph{integral
  foliated simplicial volume}, see~\cite[page~305f]{Gromov(1999a)},~\cite{Schmidt(2005)},
or~\cite[Section~2]{Frigerio-Loeh-Pagliantini_Sauer(2016)}, and of the \emph{stable
  integral simplicial volume}, see~\cite[page~709]{Frigerio-Loeh-Pagliantini_Sauer(2016)}.
The integral foliated simplicial volume gives an upper bound on the torsion growth for an
oriented closed manifold, i.e, an upper bound on
$\limsup_{i \to \infty} \;\frac{\ln\big(\bigl|\tors(H_n(M[i];\IZ))\bigr|\bigr)}{[G:G_i]}$
in the situation of
Conjecture~\ref{con:Homological_growth_and_L2-torsion_for_aspherical_manifolds},
see~\cite[Theorem~1.6]{Frigerio-Loeh-Pagliantini_Sauer(2016)}.  There are the open
questions whether for an aspherical oriented closed manifold the simplicial volume and the
integral foliated simplicial volume agree and whether for an aspherical oriented closed
manifold with residually finite fundamental group the integral foliated simplicial volume
and the stable integral simplicial volume agree, see~\cite[Question~1.2 and
Question~1.3]{Frigerio-Loeh-Pagliantini_Sauer(2016)}.  The stable integral simplicial
volume and the simplicial volume agree for aspherical oriented closed $3$-manifolds,
see~\cite[Theorem~1]{Fauser-Loeh-Moraschini-Quintanilha(2019)}.


\typeout{---- Section 6: Twisting $L^2$-invariants with finite-dimensional representations
  -----}

\section{Twisting $L^2$-invariants with finite-dimensional representations}%
\label{sec:Twisting_L2-invariants_with_finite-dimensional_representations}

In general one would like to twist $L^2$-Betti numbers and $L^2$-torsion with a
finite-dimensional representation. In this section we discuss the general case and the
technical difficulties and potential applications.  The case, where the representation is
a $1$-dimensional real representation, is much easier, since then all the technical
problems have been solved, and is very interesting for $3$-manifolds.  It is treated in
Section~\ref{sec:Twisting_L2-invariants_with_a_homomorphism_to_IR} and a reader may
directly pass to Section~\ref{sec:Twisting_L2-invariants_with_a_homomorphism_to_IR}.

A strategy to do this is discussed in~\cite{Lueck(2018)}.  Consider a group $G$ and a
$d$-dimensional complex $G$-representation $V$. Consider a $\IC G$-homomorphism
$f \colon \IC G^m \to \IC G^n$.  Choose a $\IC $-basis $B$ for the underlying complex
vector spaces $V$. (No compatibility conditions with the $G$-actions are required for this
basis.)  Then one can define a new $\IC G$-homomorphism
$\eta^G_{V,B}(f) \colon \IC G^{md} \to \IC G^{nd}$, see
Remark~\ref{Basic_idea_of_the_construction_of_eta_upper_G_(V,B)}. By applying
$L^2(G) \otimes_{\IC G} -$, we obtain a bounded $G$-equivariant operator
$\eta^{(2)}_{V,B}(f) \colon L^2(G)^{md} \to L^2(G)^{nd}$.  Analogously, we can assign to a
finite based free $\IC G$-chain complex $C_*$ a finite Hilbert $\caln(G)$-chain complex
$\eta^{(2)}_{V,B} (C_*)$.

One important question is what the relationship of the $L^2$-Betti numbers of
$\eta^{(2)}_{V,B} (C_*)$ and of $C^{(2)}_*$ are. The hope is that
\begin{eqnarray}
  b_p^{(2)}(\eta^{(2)}_{V,B} (C_*)) & = & \dim_{\IC}(V) \cdot b_p^{(2)}(C_*^{(2)})
                                          \label{L2-Betti_numbers_and_twistings}
\end{eqnarray}                                        
holds. This has interesting consequences for the behaviour of $L^2$-Betti numbers under
fibrations, see~\cite[Section~5.2]{Lueck(2018)}.  The answer to
Question~\ref{L2-Betti_numbers_and_twistings} is positive if $G$ is a torsionfree
elementary amenable group, see~\cite[Lemma~5.2]{Lueck(2018)}.

For $L^2$-torsion the following question is crucial.
\begin{question}\label{L2-torsion_and_twisting}
  Suppose that $C_*^{(2)}$ is $\det$-$L^2$-acyclic. Is then $\eta^{(2)}_{V,B} (C_*)$
  $\det$-$L^2$-acyclic?
\end{question}

\begin{remark}[$L^2$-torsion and character
  varieties]\label{rem:L2-torsion_and_character_varieties}
  Let $G$ be a group and denote by $R(G,\GL_n(\IC))$ the character variety given by group
  homomorphisms $u \colon G \to \GL_n(\IC)$.  Let $C_*$ be a $\det$-$L^2$-acyclic
  $\IC G$-chain complex. It may come from the cellular $\IZ G$-chain complex
  $C_*(\widetilde{X})$ for an appropriate $\det$-$L^2$-acyclic $CW$-complex $X$ with
  $G \cong \pi_1(X)$, for instance from a closed aspherical manifold $X$ of odd dimension
  with $G \cong \pi_1(X)$.
  
  If the answer to Question~\ref{L2-torsion_and_twisting} is positive, one could study the
  interesting function from the character variety $R(G,\GL_n(\IC))$, which assign to such
  $u \in R(G,\GL_n(\IC))$ the $L^2$-torsion of $\eta^{(2)}_{V_u,B_u} (C_*)$ for $V_u$ the
  $n$-dimensional complex $G$-representation with the obvious basis $B$ associated to $u$.
  One may ask whether the function on the character variety is continuous. This problem is
  in general wide open, but solved in some special case as we will see below.  We will
  describe some special cases, where this type of function leads to interesting results.
    
  If $\eta^{(2)}_{V_u,B_u} (C_*)$ has a gap at the spectrum at zero, then obviously the
  $L^2$-torsion of $\eta^{(2)}_{V_u,B_u} (C_*)$ is well-defined. Moreover the function
  sending $v \in R(G,\GL_n(\IC))$ to the $L^2$-torsion of $\eta^{(2)}_{V_u,B_u} (C_*)$ is
  well-defined and continuous in neighborhood of $u$. This follows form the continuity of
  the Fulgede-Kadison determinant for invertible matrices over the group von Neumann
  algebra with respect to the norm topology,
  see~\cite[Theorem~1.10~(d)]{Carey-Farber-Mathai(1997)},%
~\cite[Theorem~1~(3)]{Fuglede-Kadison(1952)}, or,~\cite[Lemma~9.14]{Lueck(2018)}. This
  is studied in more detail for a hyperbolic $3$-manifold $M$ with empty or incompressible
  torus boundary and the canonical holonomy representation
  $h \colon \pi_1(M) \to \SL_2(\IC)$ by B\'enard-Raimbault~\cite{Benard-Raimbault(2020)}.
  They actually show that this function is real analytic near $h$.
\end{remark}

\begin{remark}[Twisting $L^2$-torsion for $\det$-$L^2$-acyclic finite $CW$-complexes]%
\label{rem:Twisting_L2-torsion_for_det-L2-acyclic_finite_CW-complexes}
  Of course it is interesting to study for a $\det$-$L^2$-acyclic finite $CW$-complex its
  twisted $L^2$-torsion
  $\rho^{(2)}(X;V,B) = \rho^{(2)}(\eta^{(2)}_{V,B} (C_*(\widetilde{X})))$ for a finite
  dimensional complex $\pi_1(X)$-repre\-sen\-ta\-tion $V$ with a basis $B$ for its
  underlying vector space. The basic properties including the independence of the choice
  of $B$ are discussed in~\cite[Theorem~6.7]{Lueck(2018)}.

  Unfortunately, deciding Question~\ref{L2-torsion_and_twisting} seems to be very hard.
  The only case, where one knows that the answer is positive, is the one, where $G$ is
  finitely generated residually finite, $V$ is a $\IZ^d$-representation and $V$ is viewed
  as $G$-representation by a group homomorphism from $G$ or $\pi_1(X)$ to $\IZ^d$,
  see~\cite[Theorem~6.7]{Lueck(2018)}. There are interesting results in this setting as we
  see below, for instance if $V$ is $1$-dimensional.
\end{remark}

\begin{remark}[Unitary representations]\label{rem:unitary_representations}
  If the representation is unitary, then~\eqref{L2-Betti_numbers_and_twistings} is true
  and the answer to Question~\ref{L2-torsion_and_twisting} is positive, Moreover, we have
  $\rho^{(2)}(\eta^{(2)}_{V,B} (C_*)) = \dim_{\IC}(V) \cdot \rho^{(2)}(C_*^{(2)})$ and
  hence the twisting has no interesting effect, see~\cite[Theorem~3.1]{Lueck(2018)}.
  Hence it is crucial to consider not necessarily unitary representations.
\end{remark}

All these problems are related to the following question.  Define the \emph{regular
  Fuglede-Kadison determinant} of a morphism $f \colon U \to U$ of finitely generated
Hilbert $\caln(G)$-modules
\begin{eqnarray} \quad \quad \quad {\det}_{\caln(G)}^r(f) & := &
                                                                 \begin{cases} {\det}_{\caln(\Gamma)}(f) 
                                                                   & 
                                                                   \text{if} \; f \;\text{is injective and of determinant class;}
                                                                   \\ 
                                                                   0 & 
                                                                   \text{otherwise,}
                                                                 \end{cases}
                                                                       \label{regular_Fuglede-Kadison_determinant}
\end{eqnarray}

One should not confuse the Fuglede-Kadison determinant $\det_{\caln(G)}(f)$ and the
regular Fuglede-Kadison determinant $\det_{\caln(G)}^r(f)$ of a morphism
$f \colon U \to V$ of finitely generated Hilbert $\caln(G)$-modules,
see~\cite[Remark~8.9]{Lueck(2018)}.

For an element $x = \sum_{g \in G} \lambda_g \cdot g$ in $\IC G$ define its support
$\supp_G(x)$ to be the finite subset $\{g \in G \mid \lambda_g \not=0\}$ of $G$. For a
matrix $A= (a_{i,j}) \in M(m,n;\IC G)$ define its \emph{support} to be the finite subset
$\bigcup_{i,j} \supp_G(a_i,j)$ of $G$.  The following question is taken
from~\cite[Question~9.11]{Lueck(2018)}.

\begin{question}[Continuity of the regular determinant]%
\label{que:Continuity_of_the_regular_determinant}
  Let $G$ be a group for which there exists a natural number $d$, such that the order of
  any finite subgroup $H \subseteq G$ is bounded by $d$, e.g., $G$ is torsionfree. Let
  $S \subseteq G$ be a finite subset.  Put
  $\IC[n,S] := \{A \in M_{n,n}(\IC G) \mid \supp_G(A) \subseteq S\}$ and equip it with the
  standard topology coming from the structure of a finite-dimensional complex vector
  space.

  \begin{enumerate}

  \item\label{que:Continuity_of_the_regular_determinant:det} Is the function given by the
    regular Fuglede-Kadison determinant
    \[
      \IC[n,S] \to [0,\infty], \quad A \mapsto {\det}_{\caln(G)}^r(r_A^{(2)}\colon
      L^2(G)^n \to L^2(G)^n)
    \]
    continuous?

  \item\label{que:Continuity_of_the_regular_determinant:L2-acyclic} Consider
    $A \in \IC[S]$ such that $r_A^{(2)} \colon L^2(G)^n \to L^2(G)^n$ is a weak
    isomorphism of determinant class.  Does there exist an open neighbourhood $U$ of $A$
    in $\IC[S]$ such that for every element $B \in U$ also
    $r_B^{(2)} \colon L^2(G)^n \to L^2(G)^n $ is a weak isomorphism of determinant class?
  \end{enumerate}
\end{question}

We mention that the answer to this question is known to be negative for some finitely
presented groups for which there is no bound on the order of its finite subgroups,
see~\cite[Remark~9.12]{Lueck(2018)}.  Moreover, one cannot discard the condition about the
existence of the finite set $S$, see~\cite[Remark~9.13]{Lueck(2018)}.  The answer is
positive if $G$ is finitely generated abelian. It is possible that the answer is always
positive for a torsionfree finitely generated group $G$.

\begin{remark}[Basic idea of the construction of $\eta^G_{V,B}$]%
\label{Basic_idea_of_the_construction_of_eta_upper_G_(V,B)}
  The basic idea is the following. Let $M$ and $V$ be $\IC G$-modules. Denote by
  $(M \otimes_{\IC} V)_1$ the $\IC G$-module whose underlying vector space is
  $M \otimes_{\IC} V$ and on which $g \in G$ acts only on the first factor, i.e.,
  $g (u \otimes v) = gu \otimes v$.  Denote by $(M\otimes_{\IC} V)_d$ the $\IC G$-module
  whose underlying vector space is $M \otimes_{\IC} V$ and on which $g \in G$ acts
  diagonally, i.e., $g (u \otimes v) = gu \otimes gv$.  Note that $(M \otimes_{\IC} V)_1$
  is independent of the $G$-action on $V$ and $\IC G$-isomorphic to the direct sum of
  $\dim_{\IC}(V)$ copies of $M$, whereas $(\IC G \otimes_{\IC} M)_d$ does depend on the
  $G$-action on $M$.  We obtain a $\IC G$-isomorphism
  \[
    \xi_V(M) \colon (M \otimes_{\IC} V)_1 \xrightarrow{\cong} (M \otimes_{\IC} V)_d, \quad
    g \otimes v \mapsto g \otimes gv,
  \]
  whose inverse sends $g \otimes v$ to $g \otimes g^{-1}v$. Given a $\IC G$-homomorphism
  $f \colon \IC G^m \to \IC G^n$, we obtain a $\IC G$-homomorphism
  $(f \otimes_{\IC} \id_V)_d \colon (\IC G^m \otimes_{\IC} V)_d \to (\IC G^n \otimes_{\IC} V)_d$.  If $V$ is
  a $d$-dimensional complex representation which comes with a basis for the underlying
  complex vector space, we obtain an identification
  $(\IC G^m \otimes_{\IC} V)_1 = \IC G^{md}$ and we define $\eta^G_{V,B}(f)$ by requiring
  that the following diagram commutes
  \[
    \xymatrix@!C=11em{\IC G^{md} = (\IC G^m \otimes_{\IC} V)_1 \ar[r]^-{\xi_V(\IC G^m)}
      \ar[d]_{\eta^G_{V,B}(f)} & (\IC G^m \otimes_{\IC} V)_d\ar[d]^{f \otimes_{\IC}\id_V}
      \\
      \IC G^{nd} = (\IC G^n \otimes_{\IC} V)_1 \ar[r]_-{\xi_V(\IC G^n)} & (\IC G^n
      \otimes_{\IC} V)_d }
  \]
  More details can be found in~\cite[Section~1 and~2]{Lueck(2018)}.
\end{remark}

Some information about the equality of analytic and topological torsion for the twisted versions can be found for instance
in~\cite{Wassermann(2020)}.


\typeout{---- Section 7: Twisting $L^2$-invariants with homomorphism to $\IR$ -----}

\section{Twisting $L^2$-invariants with a homomorphism to $\IR$ }%
\label{sec:Twisting_L2-invariants_with_a_homomorphism_to_IR}

Consider a finite connected $CW$-complex $X$ and an element
$\phi \in H^1(X;\IR) = \hom(\pi_1(X),\IR)$. We call two functions
$f_0, f_1 \colon \IR^{>0} \to \IR$ \emph{equivalent} if there exists an element
$r \in \IR$ such that $f_0(t) - f_1(t) = r \cdot \ln(t)$ holds for all $t \in
\IR^{>0}$. In the sequel function $\IR^{>0} \to \IR$ is often to be understood as an
equivalence class of functions $\IR^{>0} \to \IR$. One has to interprete some statements
to be for one and hence all representatives and equality of functions means the equality
of their equivalence classes.

\begin{assumption}\label{ass:assumption_on_pi_1(X)}
  We will assume that
  \begin{itemize}

  \item The finite $CW$-complex $X$ is $\det$-$L^2$-acyclic;
  \item Its fundamental group $\pi_1(X)$ is residually finite;
  \item Its fundamental group $\pi_1(X)$ satisfies the Farrell-Jones Conjecture for
    $\IZ G$.
  \end{itemize}
\end{assumption}

\begin{remark}[Assumption~\ref{ass:assumption_on_pi_1(X)}]\label{rem:assumption}
  The reader does not need to know what the $K$-theoretic Farrell-Jones Conjecture for
  $\IZ G$ is, it can be used as a black box. The reader should have in mind that it is
  known for a large class of groups, e.g., hyperbolic groups, CAT(0)-groups, solvable
  groups, lattices in almost connected Lie groups, fundamental groups of $3$-manifolds and
  passes to subgroups, finite direct products, free products, and colimits of directed
  systems of groups (with arbitrary structure maps).  For more information we refer for
  instance to~\cite{Bartels-Farrell-Lueck(2014),
    Bartels-Lueck(2012annals),Bartels-Lueck-Reich(2008hyper), Farrell-Jones(1993a),
    Kammeyer-Lueck-Rueping(2016), Lueck-Reich(2005),Wegner(2015)}.

  In particular Assumption~\ref{ass:assumption_on_pi_1(X)} is satisfied if $X$ is an
  admissible $3$-manifold.
\end{remark}

Then from the construction of
Section~\ref{sec:Twisting_L2-invariants_with_finite-dimensional_representations} we get a
well-defined (equivalence class of) function $\IR^{>0} \to \IR$, denoted by
\begin{equation}
  \overline{\rho}^{(2)}(\widetilde{X};\phi) \colon \IR^{>0} \to \IR
  \label{rho_upper_(2)(widetilde(X);phi)}
\end{equation}
and called \emph{reduced twisted $L^2$-torsion function}. It sends $t \in \IR^{>0}$ to
$\rho^{(2)}(\widetilde{X};\IC_t)$ for the complex representation with underlying complex
vector space $\IC$ on which $g \in \pi_1(X)$ acts by multiplication with the real number
$t^{\phi(g)}$.

If $X$ is a finite not necessarily connected $CW$-complex, we require that
Assumption~\ref{ass:assumption_on_pi_1(X)} holds for each component $C$ of $X$ and we
define
\[
  \overline{\rho}^{(2)}(\widetilde{X};\phi) = \sum_{C \in \pi_0(X)}
  \overline{\rho}^{(2)}(\widetilde{C};\phi|_C).
\]

\begin{theorem}[Properties of the twisted $L^2$-torsion function]\label{the:Properties_of_the_twisted_L2-torsion_function}\
  Let $X$ be a finite $CW$-complex which satisfies
  Assumption~\ref{ass:assumption_on_pi_1(X)} and comes with an element
  $\phi \in H^1(X;\IR)$.

  \begin{enumerate}

  \item\emph{Well-definedness}\label{the:Properties_of_the_twisted_L2-torsion_function:well-defined}\\
    The function $\overline{\rho}^{(2)}(\widetilde{X};\phi)$ is well-defined;

  \item\label{the:Properties_of_the_twisted_L2-torsion_function:logarithmic_estimate}
    \emph{Logarithmic estimate}\\
    There exist constants $C \ge 0$ and $D \ge 0$, such that we get for $0 < t \le 1$
    \[
      C \cdot \ln(t) -D \le \overline{\rho}^{(2)}(\widetilde{X};\phi)(t) \le - C \cdot
      \ln(t) + D,
    \]
    and for $t \ge 1$
    \[
      - C \cdot \ln(t) -D \le \overline{\rho}^{(2)}(\widetilde{X};\phi)(t) \le C \cdot
      \ln(t) + D;
    \]

  \item\label{the:Properties_of_the_twisted_L2-torsion_function:homotopy-invariance}
    \emph{$G$-homotopy invariance}\\
    Let $Y$ be a finite $CW$-complex and let $f \colon Y \to X$ be a $G$-homotopy
    equivalence. Denote by $f^*\phi \in H^1(Y;\IR)$ the image of $\phi$ under the
    isomorphism $H^1(f;\IR) \colon H^1(X;\IR) \xrightarrow{\cong} H^1(Y;\IR)$.

    Then $Y$ satisfies Assumption~\ref{ass:assumption_on_pi_1(X)} with respect to
    $f^*\phi$ and we get
    \[
      \overline{\rho}^{(2)}(Y;f^*\phi) = \rho^{(2)}(X;\phi);
    \]

  \item\label{the:Basic_properties_of_the_reduced_L2-torsion_function_for_universal_coverings:sum_formula}
    \emph{Sum formula}\\
    Consider a cellular pushout of finite $CW$-complexes
    \[
      \xymatrix{ X_0 \ar[r]^{i_1} \ar[rd]^{j_0} \ar[d]_{i_2} & X_1 \ar[d]^{j_1}
        \\
        X_2 \ar[r]_{j_2} & X }
    \]
    where $i_1$ is cellular, $i_0$ is an inclusion of $CW$-complexes and $X$ has the
    obvious $CW$-structure coming from the ones on $X_0$, $X_1$ and $X_2$. Suppose that
    for $i =0,1,2$ the map $j_i$ is $\pi_1$-injective, i.e., for any choice of bases point
    $x_i \in X_i$ the induced map
    $\pi_1(j_i,x_i) \colon \pi_1(X_i,x_i) \to \pi_1(X,j_i(x_i))$ is injective.  Suppose we
    are given elements $\phi_i \in H^1(X_i;\IR)$ and $\phi \in H^1(X;\IR)$ such that
    $j_i^*(\phi) = \phi_i$ holds for $i = 0,1,2$.  Assume that $X_i$ for $i = 0,1,2$ and
    $X$ satisfy Assumption~\ref{ass:assumption_on_pi_1(X)}.

    Then we get
    \[
      \overline{\rho}^{(2)}(\widetilde{X};\phi) =
      \overline{\rho}^{(2)}(\widetilde{X_1};\phi_1) +
      \overline{\rho}^{(2)}(\widetilde{X_2};\phi_2) -
      \overline{\rho}^{(2)}(\widetilde{X_0};\phi_0);
    \]

  \item\label{the:Basic_properties_of_the_reduced_L2-torsion_function_for_universal_coverings:product_formula}
    \emph{Product formula}\\
    Let $Y$ be a finite connected $CW$-complex such that $\pi_1(Y)$ is residually
    finite.  Consider an element $\phi' \in H^1(X \times Y;\IR)$ such that $\phi$ is the
    image of $\phi'$ under the map $H^1(X \times Y;\IR) \to H^1(X;\IR)$ induced by the
    inclusion $X \to X \times Y, \; x \mapsto (x,y)$ for any choice of base point
    $y \in Y$.  Suppose that $X$ satisfies Assumption~\ref{ass:assumption_on_pi_1(X)}.

    Then $X \times Y$ satisfies Assumption~\ref{ass:assumption_on_pi_1(X)} with respect to
    $\phi'$ and we get
    \[
      \overline{\rho}^{(2)}(\widetilde{X \times Y}; \phi') = \chi(Y) \cdot
      \overline{\rho}^{(2)}(\widetilde{X};\phi);
    \]

  \item\label{the:Basic_properties_of_the_reduced_L2-torsion_function_for_universal_coverings:Poincare_duality}
    \emph{Poincar\'e duality}\\
    Let $X$ be a finite orientable $n$-dimensional Poincar\'e complex, e.g., a closed
    orientable manifold of dimension $n$ without boundary.  Then
    \[
      \overline{\rho}^{(2)}(\widetilde{X};\phi)(t) = (-1)^{n+1} \cdot
      \overline{\rho}^{(2)}(\widetilde{X};\phi)(t^{-1});
    \]

  \item\label{the:Basic_properties_of_the_reduced_L2-torsion_function_for_universal_coverings:(Multiplicativity)}
    \emph{Finite coverings}\\
    Let $p \colon Y \to X$ be a $d$-sheeted covering. Then $Y$ satisfies
    Assumption~\ref{ass:assumption_on_pi_1(X)} with respect to $p^*\phi$ and we get
    \[
      \overline{\rho}^{(2)}(\widetilde{Y};p^*\phi_X) = d \cdot
      \overline{\rho}^{(2)}(\widetilde{X};\phi);
    \]

  \item\label{the:Properties_of_the_twisted_L2-torsion_function:scaling} \emph{Scaling $\phi$}\\
    Let $r \in \IR$ be a real number.  Then
    \[
      \rho^{(2)}(X;r \cdot \phi)(t) = \rho^{(2)}(X;\phi)(t^r).
    \]

  \item\label{the:Properties_of_the_twisted_L2-torsion_function:value_at_t_is_0} \emph{Value for $t = 0$.}\\
    The value $\overline{\rho}^{(2)}(\widetilde{X};\phi)(0)$ is the $L^2$-torsion
    $\widetilde{\rho}(\widetilde{X})$.
  \end{enumerate}
\end{theorem}

\begin{definition}[Degree of an equivalence class of functions $\IR^{> 0} \to \IR$]%
\label{def:Degree_of_an_equivalence_class_of_functions_IR_greater_0_to_IR}
  Let $\overline{\rho}$ be an equivalence class of functions $\IR^{> 0} \to \IR$. Let
  $\rho$ be a representative.  Assume that
  $\liminf_{t \to 0+} \frac{\rho(t)}{\ln(t)} \in \IR$ and
  $\limsup_{t \to \infty} \frac{\rho(t)}{\ln(t)} \in \IR$.

  Then define the \emph{degree at zero} and the \emph{degree at infinity} of $\rho$ to be
  the real numbers
  \begin{eqnarray*}
    \deg_0(\rho) 
    & := & 
           \liminf_{t \to 0+} \frac{\rho(t)}{\ln(t)};
    \\
    \deg_{\infty}(\rho) 
    & := & 
           \limsup_{t \to \infty} \frac{\rho(t)}{\ln(t)}.
  \end{eqnarray*}
  Define the \emph{degree} of $\overline{\rho}$ to be the real number
  \begin{eqnarray*}
    \deg(\overline{\rho}) 
    & := & 
           \deg_{\infty}(\rho)  -  \deg_0(\rho)  = 
           \limsup_{t \to \infty} \frac{\rho(t)}{\ln(t)} - \liminf_{t \to 0+} \frac{\rho(t)}{\ln(t)}.
    \\
  \end{eqnarray*}
\end{definition}

Thus we can assign to a finite $CW$-complex $X$ satisfying
Assumption~\ref{ass:assumption_on_pi_1(X)} and $\phi \in H^1(X;\IR)$ its \emph{degree}
\begin{equation}
  \deg(X;\phi) := \deg(\overline{\rho}^{(2)}(\widetilde{X};\phi)) \in \IR.
  \label{degree_of_phi_and_phi}
\end{equation}

This is a new invariant with high potential although it is very hard to compute. We will
be able to relate the degree to the Thurston norm for an admissible $3$-manifold in the
next
Section~\ref{sec:degree_of_the_reduced_twisted_L2-torsion_function_and_the_Thurston_norm}.

\begin{conjecture}\label{con:l2-continuous}
The reduced twisted $L^2$-torsion function
$\overline{\rho}^{(2)}(\widetilde{X};\phi) \colon \IR^{>0} \to \IR$ is continuous.

Moreover, the $\liminf$ and $\limsup$ terms appearing in Definition~\ref{def:Degree_of_an_equivalence_class_of_functions_IR_greater_0_to_IR}
are  actually  limits $\lim$.
\end{conjecture}

Conjecture~\ref{con:l2-continuous}
has been proved for admissible $3$-manifolds $X$ by Liu~\cite[Theorem~1.2]{Liu(2017)},
where also multiplicative convexity is shown.

Moreover, for an admissible $3$-manifold $X$ the degree defines a continuous function on
$H^1(X;\IR)$, see~\cite[Theorem~6.1]{Liu(2017)}. We conjecture that this is true for every
finite $CW$-complex $X$ which satisfies Assumption~\ref{ass:assumption_on_pi_1(X)}.

Many of the results of this section are inspired by classical results on the Mahler
measure, see~\cite{Boyd(1981speculations),Boyd(1998approximation)}, which is the same as
the Fuglede-Kadison determinant in the special case $G = \IZ^d$, see~\cite[Example~3.13 on
page 128 and (3.23) on page~136]{Lueck(2002)}.


\typeout{---- Section 8: The degree of the reduced twisted $L^2$-torsion function and the
  Thurston norm -----}

\section{The degree of the reduced twisted $L^2$-torsion function and the Thurston norm}%
\label{sec:degree_of_the_reduced_twisted_L2-torsion_function_and_the_Thurston_norm}.

The following result was proved independently by
Friedl-L\"uck\cite[Theorem~0.1]{Friedl-Lueck(2019Thurston)} and by
Liu~\cite[Theorem~1.2]{Liu(2017)}. The proofs depend on the facts that both the Thurston
Geometrization Conjecture and the Virtually Fibering Conjecture are true, see
Subsection~\ref{subsec:Thurstons_Geometrization_Conjecture}
and~\ref{subsec:The_Virtual_Fibration_Conjecture}.

\begin{theorem}\label{the:main_result_introduction} 
  Let $M$ be an admissible $3$-manifold in the sense of
  Definition~\ref{def:admissible_3-manifold}.  Then we get for any element
  $\phi \in H^1(M;\IQ)$ that
  \[
    \deg(M;\phi) = - x_M(\phi),
  \]
  where the degree $\deg(M;\phi) := \deg(\overline{\rho}^{(2)}(\widetilde{M};\phi))$ has
  been defined in Section~\ref{sec:Twisting_L2-invariants_with_a_homomorphism_to_IR} and
  $x_M(\phi)$ is the Thurston norm, see
  Subsection~\ref{subsec:The_Thurston_norm_and_the_dual_Thurston_polytope}.
\end{theorem}

Actually, Friedl-L\"uck~\cite[Theorem~5.1]{Friedl-Lueck(2019Thurston)} get a much more
general result, where one can consider not only the universal covering but appropriate
$G$-coverings $G \to \overline{M} \to M$ and get estimates for the $L^2$-function for all
times $t \in (0,\infty)$ which imply the equality of the degree and the Thurston norm.


\typeout{---- Section 9: The universal $L^2$-torsion and the Thurston polytope -----}

\section{The universal $L^2$-torsion and the Thurston polytope}%
\label{sec:The_universal_L2-torsion_and_the_Thurston_polytope}


\subsection{The weak Whitehead group}%
\label{subsec:The_weak_Whitehead_group}

Next we assign to a group $G$ the weak $K_1$-groups $K_1^w(\IZ G)$,
$\widetilde{K}_1^w(\IZ G)$ and the weak Whitehead group $\Wh^w(G)$, which are variations
on the corresponding classical groups.

\begin{definition}[$K^w_1(\IZ G)$]\label{def:K_1_upper_w(ZG)}
  Define the \emph{weak $K_1$-group}
  \[
    K_1^w(\IZ G)
  \]
  to be the abelian group defined in terms of generators and relations as follows.
  Generators $[f]$ are given by of $\IZ G$-endomorphisms $f \colon \IZ G^n \to \IZ G^n$
  for $n \in \IZ, n \ge 0$ such that the induced bounded $G$-equivariant operator
  $f^{(2)} \colon L^2(G)^n \to L^2(G)^n$ is a weak isomorphism of finite Hilbert
  $\caln(G)$-modules.  If $f_1,f_2 \colon \IZ G^n \to \IZ G^n$ are $\IZ G$-endomorphisms
  such that $f_1^{(2)}$ and $f_2^{(2)}$ are weak isomorphisms, then we require the
  relation
  \[ [f_2 \circ f_1] = [f_1] + [f_2].
  \]
  If $f_0 \colon \IZ G^m \to \IZ G^m$, $f_2 \colon \IZ G^n \to \IZ G^n$ and
  $f_1 \colon \IZ G ^n \to \IZ G^m$ are $\IZ G$-maps such that $f_0^{(2)}$ and $f_2^{(2)}$
  are weak isomorphisms, then we get for the $\IZ G$-map
  \[
    f = \begin{pmatrix} f_0 & f_1 \\ 0 & f_2 \end{pmatrix} \colon \IZ G^{m + n} = \IZ G^m
    \oplus \IZ G^n \to \IZ G^m \oplus \IZ G^n
  \]
  the relation
  \[ [f] = [f_0] + [f_2].
  \]
  Let
  \[
    \widetilde{K}^w_1(\IZ G)
  \]
  be the quotient of $K_1^w(\IZ G)$ by the subgroup generated by the element
  $[- \id \colon \IZ G\to \IZ G]$.  This is the same as the cokernel of the obvious
  composite $K_1(\IZ) \to K_1^w(\IZ G)$.  Define the \emph{weak Whitehead group} of $G$
  \[
    \Wh^w(G)
  \]
  to be the cokernel of the homomorphism
  \[
    \{\sigma \cdot g \mid \sigma \in \{ \pm 1\}, g \in G\} \to K_1^w(\IZ G), \quad \sigma
    \cdot g \mapsto [r_{\sigma \cdot g} \colon \IZ G \to \IZ G].
  \]
\end{definition}

These groups are in general much larger than their classical analogues. For example, we
have $\Wh(\IZ)=0$ and
$\Wh^w(\IZ)\cong \IQ(z^{\pm 1})^{\times}/\{\pm z^n \mid n \in \IZ\}$. More generally, if
$G$ is torsionfree, the Farrell-Jones Conjecture implies $\Wh(G) = \{0\}$ and we have the
following result taken from~\cite[Theorem~0.1]{Linnell-Lueck(2018)}.

\begin{theorem}[$K_1^w(G)$ and units in $\cald(G)$]\label{the:Whw(G)_and_units_in_cald(G;IQ)}
  Let $\calc$ be the smallest class of groups which contains all free groups and is closed
  under directed unions and extensions with elementary amenable quotients.  Let $G$ be a
  torsionfree group which belongs to $\calc$.

  Then the division closure $\cald(G)$ of $\IQ G$ in $\calu(G)$ is a skew field and there
  are isomorphisms of abelian groups
  \[
    K_1^w(\IZ G) \xrightarrow{\cong} K_1(\cald(G)) \xrightarrow{\cong}
    \cald(G)^{\times}/[\cald(G)^{\times},\cald(G)^{\times}].
  \]
\end{theorem}


\subsection{The universal $L^2$-torsion}%
\label{subsec:The_universal_L2-torsion}

Given an $L^2$-acyclic finite based free $\IZ G$-chain complex $C_*$,
Friedl-L\"uck~\cite[Definition~1.7]{Friedl-Lueck(2017universal)} assign to it its
\emph{universal $L^2$-torsion},
\begin{equation}
  \rho^{(2)}_u(C_*) \in \widetilde{K}_1^w(G).
\end{equation}
It is characterized by the universal properties that
\[\rho^{(2)}_u\big(0\to \IZ G\xrightarrow{\pm \id}\IZ G\to 0\big) = 0\] and that for any
short based exact sequence $0 \to C_* \to D_* \to E_* \to 0$ of $L^2$-acyclic finite based
free $\IZ G$-chain complexes we get
$\rho^{(2)}_u(D_*) =\rho^{(2)}_u(C_*) + \rho^{(2)}_u(E_*)$, as explained
in~\cite[Definition~1.16]{Friedl-Lueck(2017universal)}. If $X$ is a $\det$-$L^2$-acyclic
finite $CW$-complex with fundamental group $\pi = \pi_1(X)$, it defines an element
\begin{equation}
  \rho^{(2)}_u(\widetilde{X}) \in \Wh^w(\pi)
\end{equation}
determined by $\rho^{(2)}_u(C_*(\widetilde{X}))$, where $C_*(\widetilde{X})$ is the
cellular $\IZ \pi$-chain complex of the universal covering $\widetilde{X}$ of $X$.

The basic properties of these invariants including homotopy invariance, sum formula,
product formula, and Poincar\'e duality are collected
in~\cite[Theorem~2.11]{Friedl-Lueck(2017universal)}. One can show for a finitely presented
group $G$, for which there exists at least one $L^2$-acyclic finite connected $CW$-complex
$X$ with $\pi_1(X) \cong G$, that every element in $\Wh^w(G)$ can be realized as
$\rho^{(2)}_u(\widetilde{Y})$ for some $L^2$-acyclic finite connected $CW$-complex $Y$
with $G \cong \pi_1(Y)$, see~\cite[Lemma~2.8]{Friedl-Lueck(2017universal)}.

The point of this new invariant is that it encompasses many other well-known invariants
such as the reduced twisted $L^2$-torsion function (which is sometimes also called
$L^2$-Alexander torsion), as explained
in~\cite[Introduction]{Friedl-Lueck(2017universal)}.  We next illustrate this by
considering the dual Thurston polytope of an admissible $3$-manifold.


\subsection{Polytopes}%
\label{subsec:polytopes}

A \emph{polytope} in a finite-dimensional real vector space $V$ is a subset which is the
convex hull of a finite subset of $V$.  An element $p$ in a polytope is called
\emph{extreme} if the implication
$p= \frac{q_1}{2} + \frac{q_2}{2} \Longrightarrow q_1 = q_2 = p$ holds for all elements
$q_1$ and $q_2$ in the polytope.  Denote by $\Ext(P)$ the set of extreme points of $P$.
If $P$ is the convex hull of the finite set $S$, then $\Ext(P) \subseteq S$ and $P$ is the
convex hull of $\Ext(P)$.  The \emph{Minkowski sum} of two polytopes $P_1$ and $P_2$ is
defined to be the polytope
\[
  P_1 + P_2 := \{p_1 +p_2 \mid p_1 \in P_1, p \in P_2\}.
\]
It is the convex hull of the set
$\{p_1 + p_2 \mid p_1 \in \Ext(P_1), p_2 \in \Ext(P_2)\}$.

Let $H$ be a finitely generated free abelian group. We obtain a finite-dimensional real
vector space $\IR \otimes_{\IZ} H$.  An \emph{integral polytope} in $\IR \otimes_{\IZ} H$
is a polytope such that $\Ext(P)$ is contained in $H$, where we consider $H$ as a lattice
in $\IR \otimes_{\IZ} H$ by the standard embedding
$H \to \IR \otimes_{\IZ} H, \; h \mapsto 1 \otimes h$.  The Minkowski sum of two integral
polytopes is again an integral polytope. Hence the integral polytopes form an abelian
monoid under the Minkowski sum with the integral polytope $\{0\}$ as neutral element.

\begin{definition}[Grothendieck group of integral polytopes]\label{def:The_Grothendieck_group_of_integral_polytope}
  Let $\calp_{\IZ}(H)$ be the abelian group given by the Grothendieck construction applied
  to the abelian monoid of integral polytopes in $\IR \otimes_{\IZ} H$ under the Minkowski
  sum.
\end{definition}

Notice that for polytopes $P_0$, $P_1$ and $Q$ in a finite-dimensional real vector space
we have the implication $P_0 + Q = P_1 + Q \Longrightarrow P_0 = P_1$,
see~\cite[Lemma~2]{Radstroem(1952)}.  Hence elements in $\calp_{\IZ}(H)$ are given by
formal differences $[P] - [Q]$ for integral polytopes $P$ and $Q$ in $\IR \otimes_{\IZ} H$
and we have $[P_0] - [Q_0] = [P_1] - [Q_1] \Longleftrightarrow P_0 + Q_1 = P_1 + Q_0$.

There is an obvious homomorphism of abelian groups $i \colon H \to \calp_{\IZ}(H)$ which
sends $h \in H$ to the class of the polytope $\{h\}$. Denote its cokernel by
\begin{eqnarray}
  \calp_{\IZ}^{\Wh}(H) 
  & = & 
        \coker\bigl(i \colon H\to \calp_{\IZ}(H)\bigr).
        \label{calp_IZ,Wh}
\end{eqnarray}
Put differently, in $\calp_{\IZ}^{\Wh}(H)$ two polytopes are identified if they are
obtained by translation with some element in the lattice $H$ from one another.

\begin{example}\label{ex:polytopes-in-r}
  An integral polytope in $\IR \otimes_{\IZ} \IZ$ is given by an interval $[m,n]$ for
  integers $m,n$ with $m \le n$. The Minkowski sum becomes
  $[m_1,n_1] + [m_2,n_2] = [m_1 + m_2, n_1 +n_2]$. One easily checks that one obtains
  isomorphisms of abelian groups
  \begin{eqnarray}
    \quad \quad \quad \calp_{\IZ}(\IZ) & \xrightarrow{\cong} & \IZ^2
                                                               \quad 
                                                               [[m,n]]  \mapsto (n-m,m); 
                                                               \label{calp_Z(Z)_is_Z2}
    \\
    \quad \quad \quad \calp_{\IZ}^{\Wh}(\IZ) & \xrightarrow{\cong} & \IZ, \quad [[m,n]] \mapsto n-m. 
                                                                     \label{calp_Z,Wh(Z)_is_Z}
  \end{eqnarray}
\end{example}

Given a homomorphism of finitely generated abelian groups $f \colon H \to H'$, we can
assign to an integral polytope $P \subseteq \IR \otimes_{\IZ} H$ an integral polytope in
$\IR \otimes_{\IZ} H'$ by the image of $P$ under
$\id_{\IR} \otimes_{\IZ} f \colon \IR \otimes_{\IZ} H \to \IR \otimes_{\IZ} H'$ and thus
we obtain homomorphisms of abelian groups
\begin{eqnarray}
  \calp_{\IZ}(f) \colon \calp_{\IZ}(H) & \to & \calp_{\IZ}(H'),
                                               \quad [P] \mapsto [\id_{\IR} \otimes_{\IZ} f(P)];
                                               \label{calp_Z(f)}
  \\
  \calp_{\IZ}^{\Wh}(f) \colon \calp_{\IZ}^{\Wh}(H) & \to & \calp_{\IZ}^{\Wh}(H'). 
                                                           \label{calp_Z,Wh(f)}
\end{eqnarray}

The elementary proof of the next lemma can be found
in~\cite[Lemma~3.8]{Friedl-Lueck(2017universal)}.

\begin{lemma}\label{lem:structure_of_polytope_group}
  Let $H$ be a finitely generated free abelian group. Then:

  \begin{enumerate}

  \item\label{lem:structure_of_polytope_group:injection} The homomorphism
    \[
      \xi \colon \calp_{\IZ}(H) \to\prod_{\phi \in \hom_{\IZ}(H,\IZ)} \calp_{\IZ}(\IZ),
      \quad [P] - [Q] \mapsto \bigl(\calp_{\IZ}(\phi)([P]- [Q])\bigr)_{\phi}
    \]
    is injective;

  \item\label{lem:structure_of_polytope_group:splitting} The canonical short sequence of
    abelian groups
    \[0\, \to \,H \,\xrightarrow{i} \,\calp_{\IZ}(H) \,\xrightarrow{\pr} \,
      \calp_{\IZ}^{\Wh}(H)\, \to\, 0\] is split exact;

  \item\label{lem:structure_of_polytope_group:divisibility} The abelian groups
    $\calp_{\IZ}(H)$ and $\calp_{\IZ}^{\Wh}(H)$ are free.  They are finitely generated
    free if and only if $H \cong \IZ$.

  \end{enumerate}
\end{lemma}

Explicit bases of the free abelian groups $\calp_{\IZ}(H)$ and $\calp_{\IZ}^{\Wh}(H)$ are
constructed by Funke~\cite{Funke(2016)}.


\subsection{The polytope homomorphism and the $L^2$-torsion polytope}%
\label{subsec:The_polytope_homomorphism_and_the_L2-torsion_polytope}

Given a torsionfree group $G$ that satisfies the Atiyah
Conjecture, the \emph{polytope homomorphism}
\begin{equation}
  \IP \colon \Wh^w(\IZ G) \to \calp_{\IZ}^{\Wh}(H_1(G)_f)
  \label{polytop_homomorphism}
\end{equation}
is constructed in~\cite[Section~3.2~3.2]{Friedl-Lueck(2017universal)}.

\begin{definition}[$L^2$-torsion polytope]\label{L2-torsion_polytope}
  Let $X$ is an $L^2$-acyclic finite $CW$-complex such that $\pi_1(X)$ is torsionfree and
  satisfies the Atiyah Conjecture.  The \emph{$L^2$-torsion
    polytope}
  \[
    P(\widetilde{X}) \in \calp_{\IZ}^{\Wh}(H_1(G)_f)
  \]
  is defined to be the negative of the image of the universal $L^2$-torsion
  $\rho^{(2)}_u(\widetilde{X})$ defined in
  Subsection~\ref{subsec:The_universal_L2-torsion} under the polytope
  homomorphism~\eqref{polytop_homomorphism}.
\end{definition}

Note that we abuse language here a little bit, the $L^2$-torsion polytope is a formal
difference of integral polytopes and not itself a polytope.


\subsection{The dual Thurston polytope and the $L^2$-torsion polytope}%
\label{subsec:The_dual_Turston_polytope_and_the_L2-torsion_polytope}

Of particular interest is the composition of the universal torsion with the polytope
homomorphism. For example let $M$ be an admissible $3$-manifold that is not a closed graph
manifold. Then we obtain a well-defined element
\[
  P(\widetilde{M})\,:=\, \IP(\rho_u(\widetilde{M}))\,\in \, \calp_{\IZ}^{\Wh}(H_1(M)_f).
\]

Recall from
Theorem~\ref{the:Status_of_the_Atiyah_Conjecture}~\ref{the:Status_of_the_Atiyah_Conjecture:3-manifold_not_graph}
the fundamental group of an admissible $3$-manifold $M$ satisfies the Atiyah
Conjecture and hence its $L^2$-torsion polytope
$P(\widetilde{X}) \in \calp_{\IZ}^{\Wh}(H_1(G)_f)$ is defined.  The next result is taken
from~\cite[Theorem~3.7]{Friedl-Lueck(2017universal)}.

\begin{theorem}[The dual Thurston polytope and the $L^2$-torsion polytope]%
\label{the:dual_Thurston_polytope_and_the_L2-torsion_polytope_new}
  Let $M$ be an admissible $3$-manifold.  Then
  \begin{eqnarray*} [T(M)^*] & = & 2 \cdot P(\widetilde{M})\,\, \in \,\,
    \calp_{\IZ}^{\Wh}(H_1(M)_f).
  \end{eqnarray*}
\end{theorem}


\typeout{------------- Section 10: Profinite completion of the fundamental group of a
  $3$-manifold ----------}

\section{Profinite completion of the fundamental group of a $3$-manifold}%
\label{sec:Profinite_completion_of_the_fundamental_group_of_a_3-manifold}

We can associate to a (discrete) group \emph{its profinite completion} defined as
\begin{eqnarray}
  \widehat{G}  & := & \operatorname{invlim}_{N} G/N
                      \label{widehat(G)}
\end{eqnarray}
where $N$ runs through normal subgroups $N$ of $G$ with finite index $[G:N]$.  The inverse
limit $\widehat{G}$ is a compact, totally disconnected group.  The canonical group
homomorphism $i \colon G \to \widehat{G}$ has dense image.

\begin{definition}[Profinitely rigid]\label{def:profinitely_rigid}
  A finitely generated residually finite group $G$ is \emph{profinitely rigid} if for
  every finitely generated residually finite group $K$ with
  $\widehat{K} \cong \widehat{G}$ we have $K \cong G$.
\end{definition}

It makes no difference whether $\widehat{K} \cong \widehat{G}$ means abstract isomorphism
of groups or topological group isomorphism, see Nikolov and
Segal~\cite[Theorem~1.1]{Nikolov-Segal(2007profinite_I)}.

Recall that an admissible $3$-manifold $N$ is topologically rigid in the sense that any
other $3$-manifold $N$ with $\pi_1(N) \cong \pi_1(M)$ is homeomorphic to $M$, see
Subsection~\ref{subsec:Topological_rigidity}. So one may ask whether for two admissible
$3$-manifolds the following three assertions are equivalent
\begin{itemize}
\item $\widehat{\pi_1(M)} \cong \widehat{\pi_1(N)}$;
\item $\pi_1(M)\cong \pi_1(N)$;
\item $M$ and $N$ are homeomorphic.
\end{itemize}
 
To the author's knowledge profinite rigidity of fundamental groups of hyperbolic closed
$3$-manifolds, even among themselves, is an open question. Examples of hyperbolic closed
3-manifolds, whose fundamental groups are profinite rigid in the absolute sense, are
constructed in~\cite{Bridson-McReynolds-Reid-Spliler(2020)}. A weaker but still open
problem is the following which is equivalent to~\cite[Conjecture~6.33]{Kammeyer(2019)}.

\begin{conjecture}[Volume and profinitely rigidity]\label{con:Volume_and_profinitely_rigidity}
  Let $M$ and $N$ be admissible closed $3$-manifolds.  Then
  $\widehat{\pi_1(M)} \cong \widehat{\pi_1(N)}$ implies
  $\rho^{(2)}(\widetilde{M}) = \rho^{(2)}(\widetilde{N})$.
\end{conjecture}

Recall that for two hyperbolic $3$-manifolds we have
$\rho^{(2)}(\widetilde{M}) = \rho^{(2)}(\widetilde{M}) \Longleftrightarrow \vol(M) =
\vol(N)$, see Theorem~\ref{the:L2-torsion_of_irreducible_3-manifold}, and there are up to
diffeomorphism only finitely many hyperbolic $3$-manifolds with the same volume.

Liu~\cite[Theorem~1.1]{Liu(2020)} has shown that among the class of finitely generated
$3$-manifold groups, every finite-volume hyperbolic $3$-manifold group is profinitely
almost rigid, where a group G is \emph{profinitely almost rigid} among a class of groups
$\calc$, if there exist finitely many groups in $\calc$, such that any group in $\calc$
that is profinitely isomorphic to $G$ is isomorphic to one of those groups. Seifert manifolds and graph manifolds
have been treated in~\cite{Wilkes(2017Seifert),Wilkes(2018Journal_of_Algebra)}.

The proof of the following result can be found in~\cite[Satz~6.34]{Kammeyer(2019)}.

\begin{theorem}\label{ref:homological_growth_versus_profinitely_rigid}
  If Conjecture~\ref{con:Homological_growth_and_L2-torsion_for_aspherical_manifolds} holds
  in dimension $3$, Conjecture~\ref{con:Volume_and_profinitely_rigidity} is true.
\end{theorem}

The first $L^2$-Betti number is profinite among finitely presented
residually finite groups, i.e., the first $L^2$-Betti numbers of two 
finitely presented residually finite groups agree if the profinite completion of these two groups are isomorphic,
see~\cite[Corollary~3.3]{Bridson-Conder-Reid(2016)}. Higher $L^2$-Betti numbers are not
profinite rigid among finitely presented  residually finite groups, see~\cite[Theorem~1]{Kammeyer-Sauer(2020)}.

For more information about profinite rigidity we refer for instance
to~\cite[Section~6.7]{Kammeyer(2019)}.


\typeout{------------------------ Section: Miscellaneous ----------------------------}

\section{Miscellaneous}\label{sec:Miscellaneous}

We can also use this approach to assign formal differences of polytopes to many other
groups, e.g., free-by-cyclic groups and two-generator one-relator groups. These examples
are discussed in more details in~\cite{Friedl-Lueck-Tillmann(2019)}, where further references to the literature can be found.

Finally let $G$ be any group that admits a finite model for $BG$ and that satisfies the
Atiyah Conjecture and let $f\colon G\to G$ be a monomorphism. Then we can associate to
this monomorphism the polytope invariant of the corresponding ascending HNN-extension. If
$G=F_2$ is the free group on two generators this polytope invariant has been studied by
Funke-Kielak~\cite{Funke-Kielak(2018)}. We hope that this invariant of monomorphisms of
groups will have other interesting applications.

Bieri-Neumann-Strebel invariants are related to polytopes and $L^2$-invariants, see
for instance~\cite{Kielak(2020)}.

There are further interesting connections between $L^2$-invariants and group theory,
orbit equivalence and von Neumann algebras 
which we cannot cover here, see for instance~\cite{Lueck(2002)}.




\end{document}